\documentclass[11pt]{amsart}
\usepackage{amssymb,amsmath,amsthm,amssymb,amsfonts,epsfig,color,graphicx,enumerate}
\usepackage[a4paper,margin=2.6truecm]{geometry}

\newcommand{\be}{\begin{equation}}
\newcommand{\ee}{\end{equation}}

\newtheorem{lm}{Lemma}[section]
\newtheorem{teo}[lm]{Theorem}
\newtheorem*{theoA}{Theorem A}
\newtheorem*{theoB}{Theorem B}
\newtheorem{coro}[lm]{Corollary}
\newtheorem{prop}[lm]{Proposition}
\theoremstyle{definition}
\newtheorem{oss}[lm]{Remark}
\newtheorem*{ack}{Acknowledgements}

\numberwithin{equation}{section}

\title{Improved energy bounds for Schr\"odinger operators}

\author[Brasco]{Lorenzo Brasco}
\address{Aix-Marseille Universit\'e, CNRS, Centrale Marseille, I2M, UMR 7373, 13453 Marseille, France}
\email{lorenzo.brasco@univ-amu.fr}
\author[Buttazzo]{Giuseppe Buttazzo}
\address{Dipartimento di Matematica, Universit\`a di Pisa, Largo B. Pontecorvo 5, 56126 Pisa, ITALY}
\email{buttazzo@dm.unipi.it}
\date{14th July 2014, Marseille et Pise}

\begin{document}

\maketitle

\begin{abstract}
Given a potential $V$ and the associated Schr\"odinger operator $-\Delta+V$, we consider the problem of providing sharp upper and lower bound on the energy of the operator. It is known that if for example $V$ or $V^{-1}$ enjoys suitable summability properties, the problem has a positive answer. In this paper we show that the corresponding isoperimetric-like inequalities can be improved by means of quantitative stability estimates.
\end{abstract}

\tableofcontents

%%%%%%%%%%%%%%%%%%%%%%%%%%%%%%%%%%%%%%%%%%%%%%%%%%

\section{Introduction}

\subsection{Aim of the paper}
Let $\Omega\subset\mathbb{R}^N$ be an open set not necessarily with finite measure (it could be $\Omega=\mathbb{R}^N$) and $V:\Omega\to\mathbb{R}$ be a potential. We consider the associated Schr\"odinger operator $-\Delta+V$ defined on the homogeneous Sobolev space $W^{1,2}_0(\Omega)$. The latter is the closure of $C^\infty_0(\Omega)$ with respect to the norm
\[
\|u\|_{W^{1,2}_0(\Omega)}:=\left(\int_\Omega|\nabla u|^2\,dx\right)^{1/2}.
\]
We also denote by $W^{-1,2}(\Omega)$ its dual space and by $\langle\cdot,\cdot\rangle$ the duality pairing between $W^{1,2}_0(\Omega)$ and $W^{-1,2}(\Omega)$.
In this paper we are concerned with the following problem: given a source term $f\in W^{-1,2}(\Omega)$, find lower and upper bounds on the {\it energy} of the relevant Schr\"odinger operator, i.e. 
\[
\mathcal{E}_f(V)=-\frac{1}{2}\, \int_\Omega |\nabla u_V|^2\, dx-\frac{1}{2}\, \int_\Omega V\, u^2_V\, dx.
\]
Here the state function $u_V$ is a $W^{1,2}_0(\Omega)$ solution of
\[
-\Delta u+V\, u=f,\qquad \mbox{ in }\Omega.
\]
In the recent paper \cite{BGRV}, this problem has been solved for summable potentials or for {\it confining potentials}, i.e. for potentials blowing-up at infinity such that $1/V$ enjoys some summability properties. For example, the harmonic--like potential
\[
V=\left(\delta^2+|x|^2\right)^{\gamma/2},\qquad \delta>0,
\] 
belongs to this class, for suitable $\gamma>0$. In order to provide a deeper insight into the scopes of this work, it is useful to briefly recall some of the results in \cite{BGRV}. In that paper it has been shown that $\mathcal{E}_f(V)$ can be universally bounded {\it from above} in the class (see \cite[Proposition 5.1]{BGRV})
\[
\mathcal{V}_1=\left\{V\, :\, \int_\Omega |V|^p\le 1\right\},
\]
and {\it from below} in the class (see \cite[Proposition 5.4]{BGRV})
\[
\mathcal{V}_2=\left\{V\ge 0\, :\, \int_\Omega V^{-p}\le 1\right\}.
\]
The value $1$ above plays no special role and can be replaced by any constant $c>0$. Indeed, one can show that there exist two potentials $V_0$ and $U_0$ such that
\[
\int_\Omega |V_0|^p\, dx=1=
\int_\Omega |U_0|^{-p}\, dx,
\]
and
\begin{equation}
\label{isop_above}
\mathcal{E}_f(V)\le \mathcal{E}_f(V_0),\qquad \mbox{ for every } V\in\mathcal{V}_1,
\end{equation}
\begin{equation}
\label{isop_below}
\mathcal{E}_f(V)\ge \mathcal{E}_f(U_0),\qquad \mbox{ for every } V\in\mathcal{V}_2.
\end{equation}
In both the estimates \eqref{isop_above} and \eqref{isop_below}, the equality sign holds if and only if $V=V_0$ and $V=U_0$, respectively.
Moreover, the extremal potentials $V_0$ and $U_0$ can be characterized in terms of the solutions $v_0$ and $u_0$ of the semilinear non-autonomous PDEs
\[
-\Delta v_0+c\, v_0^{(p+1)/(p-1)}=f\qquad \mbox{ and }\qquad -\Delta u_0+c\, u_0^{(p-1)/(p+1)}=f,
\]
through the relations
\[
V_0=\left(\int_\Omega|v_0|^{2\,p/(p-1)}\,dx\right)^{-1/p} |v_0|^{2/(p-1)}\qquad \mbox{ and }\qquad U_0=\left(\int_\Omega|u_0|^{2p/(p+1)}\,dx\right)^{1/p}|u_{0}|^{-2/(p+1)}.
\]
For completeness, we point out that a special class of potentials from the sets $\mathcal{V}_1$ and $\mathcal{V}_2$ are given respectively by
\begin{equation}
\label{insiemi}
V(x)=\begin{cases}
|E|^{-1/p},& x\in E,\\
0,& \mbox{otherwise}
\end{cases}\qquad \mbox{ and }\qquad V(x)=\begin{cases}
|E|^{1/p},& x\in E,\\
+\infty,& \mbox{otherwise}, 
\end{cases}
\end{equation}
where $E\subset\Omega$ is an open set and $|E|$ denotes its $N-$dimensional Lebesgue measure. If we suppose for simplicity that $|E|=1$, the corresponding operators are given by
\[
W^{1,2}_0(\Omega)\ni u\mapsto -\Delta u+u\cdot 1_E\qquad \mbox{ and }\qquad W^{1,2}_0(E) \ni u\mapsto -\Delta u+u.
\] 
Thus from \eqref{isop_above} and \eqref{isop_below} we get
\[
\min_{u\in W^{1,2}_0(\Omega)}\left[\frac{1}{2}\, \int_\Omega |\nabla u|^2\,dx +\frac{1}{2}\, \int_E |u|^2\, dx-\langle f,u\rangle\right]<\mathcal{E}_f(V_0),
\]
and
\[
\min_{u\in W^{1,2}_0(E)}\left[\frac{1}{2}\, \int_E |\nabla u|^2\,dx +\frac{1}{2}\, \int_E |u|^2\, dx-\langle f,u\rangle\right]>\mathcal{E}_f(U_0),
\]
for every $E\subset\Omega$ with $|E|=1$. Observe that for the second problem the set $\Omega$ simply acts as a {\it design region} where the admissible domains $E$ have to be contained.
\vskip.2cm
The problem of finding sharp bounds on energetical quantities linked to a Schr\"odinger operator is quite classical, with many studies devoted to the {\it ground state energy} or {\it first eigenvalue}
\[
\lambda_1(V)=\min_{u\in W^{1,2}_0(\Omega)}\left\{\int_{\Omega} |\nabla u|^2\, dx+\int_\Omega V\, u^2\, dx\, :\, \|u\|_{L^2(\Omega)}=1 \right\}.
\]
For example, the pioneering paper \cite{Ke} by Keller considers the problem of finding sharp lower bounds on $\lambda_1(V)$ in the class $\mathcal{V}_1$, in the case of space dimension $N=1$ and $\Omega=\mathbb{R}$. Related problems have been considered by Ashbaugh and Harrell in \cite{AH} in higher dimensions. There is a vast literature on the subject, considering optimal bounds for other spectral quantities, like the {\it first excited state} or {\it second eigenvalue} $\lambda_2(V)$ and the {\it fundamental gap} $\lambda_2(V)-\lambda_1(V)$. We also mention the recent paper \cite{BBV}, where the case of successive eigenvalues $\lambda_k(V)$ for $k\ge 2$ is considered, for the non-compact case of $\Omega=\mathbb{R}^N$.
Actually this kind of problems is even older: indeed, we observe that for potentials of the second form in \eqref{insiemi}, we have
\[
\lambda_1(V)=\min_{u\in W^{1,2}_0(E)}\left\{\int_{E} |\nabla u|^2\, dx+\int_E |u|^2\, dx :\, \|u\|_{L^2(E)}=1 \right\}=\lambda_1(E)+1,
\]
where $\lambda_1(E)$ now stands for the first eigenvalue of $-\Delta$ with Dirichlet boundary conditions on $\partial E$. Thus the celebrated {\it Faber-Krahn inequality} (see \cite[Chapter 3]{hen06})
\[
\lambda_1(E)\ge \lambda_1(B),\qquad \mbox{ with }B \mbox{ any ball such that } |B|=|E|=1, 
\]
can be seen as a particular instance of these problems. A more general overview on optimization problems of spectral type can be found in \cite{Bu} and \cite{hen06}, to which we refer the interested reader. We wish to point out that on the contrary the case of the energy $\mathcal{E}_f(V)$ appears to be less investigated.

\subsection{Main results}

In this paper, we improve the previous sharp bounds \eqref{isop_above} and \eqref{isop_below} on the energy of a Schr\"odinger operator, by means of a quantitative stability result. In other words, we will prove that the energy gap $\mathcal{E}_f(V_0)-\mathcal{E}_f(V)$ or $\mathcal{E}_f(V)-\mathcal{E}_f(U_0)$ controls the deviation from optimality of a potential $V$. Thus it is possible to add a reminder term in the right-hand side of \eqref{isop_above} and \eqref{isop_below}, which measures the distance of a generic potential $V$ from $V_0$ or $U_0$. The relevant results are summarized in the following couple of theorems, which represent the main results of the paper. We refer the reader to Sections \ref{sec:stabmax} and \ref{sec:stabmin} for the precise statements.
\begin{theoA}[Stability of the maximizer]
Let $1<p<\infty$. There exists a constant $\sigma_1>0$ such that for every $V\in\mathcal{V}_1$ we have
\[
\left\{\begin{array}{lc}\mathcal{E}_f(V)\le \mathcal{E}_f(V_0)-\sigma_1\, \|V-V_0\|^2_{L^p(\Omega)},& \mbox{ if } p\ge 2,\\
&\\
\mathcal{E}_f(V)\le \mathcal{E}_f(V_0)-\sigma_1\, \Big\||V|^{p-2}\, V-|V_0|^{p-2}\, V_0\Big\|^2_{L^{p'}(\Omega)},& \mbox{ if } 1<p<2.
\end{array}
\right.
\]
\end{theoA}
In the case of inequality \eqref{isop_below} we need to distinguish between the case $|\Omega|<+\infty$ and $|\Omega|=+\infty$.
\begin{theoB}[Stability of the minimizer]
Let $\Omega\subset\mathbb{R}^N$ be an open set such that $|\Omega|=+\infty$. Let $1<p<\infty$, $r>N/2$, and let $f\in L^{r}(\Omega)$ be a function decaying to $0$ at infinity as $O(|x|^{-\alpha})$ with $\alpha>1+N/2$. Then there exist a constant $\sigma_2>0$ and an exponent $\beta=\beta(p)>2$ such that for every $V\in\mathcal{V}_2$ we have
\[
\mathcal{E}_f(V)\ge \mathcal{E}_f(U_0)+\sigma_2\, \left\|\frac{1}{V}-\frac{1}{U_0}\right\|^\beta_{L^p(\Omega)}.
\]
If $|\Omega|<+\infty$, the same result holds for every $f\in W^{-1,2}(\Omega)$ without any additional hypothesis.
\end{theoB}
Stability results of this type have attracted an increasing interest in recent years. As a non-exhaustive list of works on the subject, we point out for example \cite{CicLeo} and \cite{FMP} dealing with the classical isoperimetric inequality, the papers \cite{BraDePVel,BraDepRuf,BraPra,HN} and \cite{Me} concerning sharp bounds for eigenvalues of the Laplacian and \cite{BE,CF,CFMP} about quantitative versions of the Sobolev and Gagliardo-Nirenberg inequalities with sharp constant.
\par
Among these papers, the recent one \cite{CFL} is very much related with the subject here considered. 
In \cite{CFL} a quantitative stability estimate for $\lambda_1(V)$ is proved, for potentials belonging to the class (here $r>N/2$)
\[
\mathcal{V}'=\left\{V\, :\, \int_{\mathbb{R}^N} |V_-|^{r}\le 1\right\},
\]
where $V_-$ is the negative part of $V$. In this case $\lambda_1(V)$ admits a sharp lower bound, corresponding to the negative potential
\[
W_0=-\left(\int_{\mathbb{R}^N} |w_0|^{2\,r/(r-1)}\right)^{-1/r}\,|w_0|^{2\,r/(r-1)},
\]
where $w_0$ is an extremal of the Gagliardo-Nirenberg inequality
\[
\left(\int_{\mathbb{R}^N} |u|^{2\,r/(r-1)}\, dx\right)^{(r-1)/r}\le C\, \left(\int_{\mathbb{R}^N} |\nabla u|^2\, dx\right)^{1-\vartheta}\,\left(\int_{\mathbb{R}^N} |u|^2\, dx\right)^\vartheta.
\]
The parameter $0<\vartheta=\vartheta(N,r)<1$ above is uniquely determined by scaling invariance.

\subsection{Plan of the paper} Along all the paper, for the sake of simplicity, we assume $N\ge 3$. In this way the Sobolev exponent $2^*$ is finite; all the results also apply, with minor modifications, to the cases $N=1$ and $N=2$ by using the corresponding Sobolev embeddings.
\par
We start with Section \ref{sec:prel} where we fix the main notations and prove some basic results which will be used throughout the whole paper. Then Section \ref{sec:max} is concerned with maximization problems for $\mathcal{E}_f$, under a constraint on the $L^p$ norm of the admissible potentials. The relevant quantitative stability result (Theorem \ref{teo:max_stab}) is then considered in Section \ref{sec:stabmax}. Minimization problems are adressed in Section \ref{sec:min}, while the last section of the paper contains the corresponding stability result (Theorem \ref{teo:min_stab}). Finally, a self-contained Appendix on sharp decay estimates for finite energy solutions of
\[
-\Delta u+c\, u^{q-1}=f,\qquad \mbox{ for } c>0,\quad 1<q<2,
\]
concludes the paper (Theorem \ref{teo:u0}).

%%%%%%%%%%%%%%%%%%%%%%%%%%%%%%%%%%%%%%%%%%%%%%%%%%
\section{Preliminaries}\label{sec:prel}

%Given an open set $\Omega\subset\mathbb{R}^N$ we denote by $W^{1,2}_0(\Omega)$ the closure of $C^\infty_0(\Omega)$ with respect to the norm
%\[
%\|u\|_{W^{1,2}_0(\Omega)}=\left(\int_\Omega|\nabla u|^2\,dx\right)^{1/2}.
%\]
%We denote by $W^{-1,2}(\Omega)$ its topological dual and by $\langle\cdot,\cdot\rangle$ the duality pairing between $W^{1,2}_0(\Omega)$ and $W^{-1,2}(\Omega)$. 
In the paper we mainly focus on the following three model cases:
\begin{itemize}
\item $\Omega=\mathbb{R}^N$;
\vskip.2cm
\item $\Omega=\omega\times\mathbb{R}$, with $\omega\subset\mathbb{R}^{N-1}$ open set with finite Lebesgue measure ({\it waveguide});
\vskip.2cm
\item $\Omega\subset\mathbb{R}^N$ with finite Lebesgue measure ({\it compact case}).
\end{itemize}
For $N\ge 3$, we define
\[
2^*=\frac{2\,N}{N-2}.
\]
The following embedding properties of $W^{1,2}_0(\Omega)$ are well known.
\begin{prop}\label{prop:embed}
Let $N\ge3$; then
\begin{enumerate}
\item[i)] if $\Omega=\mathbb{R}^N$, we have the continuous embedding $W^{1,2}_0(\Omega)\hookrightarrow L^{2^*}(\Omega)$, but $W^{1,2}_0(\Omega)\not\subset L^s(\Omega)$ for $s\not =2^*$;
\vskip.2cm
\item[ii)]if $\Omega=\omega\times\mathbb{R}$ is a waveguide, we have the continuous embedding $W^{1,2}_0(\Omega)\hookrightarrow L^s(\Omega)$ for every $2\le s\le 2^*$;
\vskip.2cm
\item[iii)]if $|\Omega|<+\infty$, we have the continuous embedding $W^{1,2}_0(\Omega)\hookrightarrow L^s(\Omega)$ for every $0<s\le 2^*$. Moreover, this is compact for $0<s<2^*$.
\end{enumerate} 
\end{prop}
In what follows, for $N\ge 3$ we set
\begin{equation}
\label{talenti}
T_N:=\inf_{u\in W^{1,2}_0(\mathbb{R}^N)}\left\{\int_{\mathbb{R}^N}|\nabla v|^2\,dx\ :\ \|v\|_{L^{2^*}(\mathbb{R}^N)}=1\right\}<+\infty.
\end{equation}
This infimum is finite by Proposition \ref{prop:embed} and attained on $W^{1,2}_0(\mathbb{R}^N)$, see for example \cite{Ta}. 
\vskip.2cm\noindent
Let $f\in W^{-1,2}(\Omega)$; for every potential $V$ belonging to the admissible class
\[
\mathcal{V}=\left\{V:\Omega \to (-\infty,+\infty]\,:\, V \mbox{ Borel measurable},\, \|V_-\|_{L^{N/2}(\Omega)}<T_N\right\},
\]
we define its energy by
\begin{equation}
\label{problema}
\mathcal{E}_f(V)=\min_{u\in W^{1,2}_0(\Omega)}\frac{1}{2}\int_\Omega |\nabla u|^2\,dx+\frac{1}{2}\int_\Omega Vu^2\,dx-\langle f,u\rangle.
\end{equation}
\begin{prop}
For every $V\in\mathcal{V}$ the minimization problem \eqref{problema} admits a solution $u_V\in W^{1,2}_0(\Omega)$. Moreover the energy inequality
\begin{equation}
\label{energyestimate}
\|u_V\|_{W^{1,2}_0(\Omega)}\le \frac{T_N}{T_N-\|V_-\|_{L^{N/2}(\Omega)}}\,\, \|f\|_{W^{-1,2}(\Omega)}. 
\end{equation}
holds.
\end{prop}

\begin{proof}
At first we observe that, taking $u=0$ gives $\mathcal{E}_f(V)\le0$. Moreover, since $V\in\mathcal{V}$, the energy functional is bounded from below, because
\begin{equation}
\label{ultimora}
\int_\Omega V\,u^2\,dx\ge-\int_\Omega V_-\,u^2\,dx
\ge-\|V_-\|_{L^{N/2}(\Omega)}\,\|u\|^2_{L^{2^*}(\Omega)}
\ge -\,\frac{\|V_-\|_{L^{N/2}(\Omega)}}{T_N}\int_\Omega|\nabla u|^2\,dx,
\end{equation}
and thus
\begin{equation}\label{dalbasso}
\begin{split}
\frac{1}{2}\int_\Omega|\nabla u|^2\,dx+&\frac{1}{2}\int_\Omega V\,u^2\,dx-\langle f,u\rangle\\
&\ge\frac{1}{2}\left(1-\frac{\|V_-\|_{L^{N/2}(\Omega)}}{T_N}-\delta\right)\int_\Omega|\nabla u|^2\,dx-\frac{1}{2\,\delta}\,\|f\|^2_{W^{-1,2}(\Omega)},
\end{split}
\end{equation}
for every $0<\delta\le 1$. Thus $\mathcal{E}_f(V)$ is finite. Let now $\{u_n\}_{n\in\mathbb{N}}\subset W^{1,2}_0(\Omega)$ be a minimizing sequence; we can assume that
\[
\frac{1}{2}\int_\Omega|\nabla u_n|^2\,dx+\frac{1}{2}\int_\Omega V\,u_n^2\,dx-\langle f,u_n\rangle\le \mathcal{E}_f(V)+1.
\] 
By \eqref{dalbasso}, if we take $\delta\ll 1$ the sequence $\{u_n\}_{n\in\mathbb{N}}$ is bounded in $W^{1,2}_0(\Omega)$, so $u_n$ weakly converges (up to a subsequence) in $W^{1,2}_0(\Omega)$ to a function $u\in W^{1,2}_0(\Omega)$. Moreover, by the compact Sobolev embedding of Proposition \ref{prop:embed} iii), we have strong convergence in $L^s(\Omega')$, for every (smooth) $\Omega'\Subset\Omega$ and every $1\le s<2^*$. In particular, $u_n$ converges almost everywhere (up to a subsequence) in $\Omega$ to $u$. Also observe that still by Proposition \ref{prop:embed} we have weak convergence of $u_n^2$ to $u^2$ in $L^{2^*/2}(\Omega)$. By using this and the Fatou Lemma, we get
\[
\liminf_{n\to\infty} \int_\Omega V\, u_n^2\, dx=\liminf_{n\to\infty}\left[\int_\Omega V_+\, u_n^2\, dx-\int_{\Omega} V_-\, u_n^2\, dx\right]\ge \int_\Omega V\, u^2\, dx.
\]
Finally, the weak lower semicontinuity of the norm implies
\[
\begin{split}
\frac{1}{2}\int_\Omega|\nabla u|^2\,dx&+\frac{1}{2}\int_\Omega V\,u^2\,dx-\langle f,u\rangle\\
&\le\liminf_{n\to\infty}\left[\frac{1}{2}\int_\Omega|\nabla u_n|^2\,dx
+\frac{1}{2}\int_\Omega V\,u_n^2\,dx-\langle f,u_n\rangle\right]=\mathcal{E}_f(V),
\end{split}
\]
which gives the existence of a minimizer $u_V$. 
\vskip.2cm\noindent
This function $u_V$ satisfies the Euler-Lagrange equation
\[
\int_\Omega\langle\nabla u,\nabla\varphi\rangle\,dx+\int_\Omega V\,u\,\varphi\,dx=\langle f,\varphi\rangle,
\]
for every $\varphi\in C^\infty_0(\Omega)\cap L^2(\Omega;V)$, where for $V\in\mathcal{V}$ we set
\[
L^2(\Omega;V)=\left\{\varphi\ :\ \int_\Omega |V|\,\varphi^2\,dx<+\infty \right\}.
\] 
By density, the previous equation holds for every $\varphi\in W^{1,2}_0(\Omega)\cap L^2(\Omega;V)$.
By taking $u_V$ as a test function and then appealing to
\[
|\langle f,u\rangle|\le\frac{1}{2\,\delta}\,\|f\|^2_{W^{-1,2}(\Omega)}+\frac{\delta}{2}\,\|u\|^2_{W^{1,2}_0(\Omega)},
\]
we have the estimate
%\begin{equation}\label{energyestimate}
\[
\left(1-\frac{\delta}{2}\right)\,\int_\Omega|\nabla u_V|^2\,dx+\int_\Omega V\,u^2_V\,dx\le\frac{1}{2\,\delta}\|f\|^2_{W^{-1,2}(\Omega)}.
\]
By using \eqref{ultimora} and choosing 
\[
\delta=1-\frac{\|V_-\|_{L^{N/2}(\Omega)}}{T_N},
\]
we get \eqref{energyestimate}.
%\end{equation}
\end{proof}
\begin{oss}
If $\|V_-\|_{L^{N/2}}\ge T_N$, in general problem \eqref{problema} is not well-posed. Indeed, take $\Omega=\mathbb{R}^N$ and indicate by $U\in W^{1,2}_0(\mathbb{R}^N)$ a positive function such that
\[
\int_{\mathbb{R}^N} |\nabla U|^2\, dx=T_N\qquad \mbox{ and }\qquad \int_{\mathbb{R}^N} |U|^{2^*}\, dx=1.
\]
We then take
\[
V=-T_N\, U^{4/(N-2)},
\]
and $f\in L^{(2^*)'}(\mathbb{R}^N)$ such that
\[
\int_{\mathbb{R}^N} f\, U\, dx>0.
\]
By evaluating the functional in \eqref{problema} on the sequence $u_n=n\, U$, we get
\[
\frac{1}{2}\int_\Omega |\nabla u_n|^2\,dx+\frac{1}{2}\int_\Omega V\,u^2_n\,dx-\langle f,u_n\rangle=-n\,\langle f,U\rangle,
\]
thus the functional is unbounded from below.
\end{oss}

\begin{lm}\label{lm_L2V}
Let $V_1,V_2\in\mathcal{V}$ be two admissible potentials and let $u_1,u_2\in W^{1,2}_0(\Omega)$ be solutions of \eqref{problema}. Then
\begin{equation}
\label{L2V}
\left|\int_\Omega V_1\,u^2_1\,dx-\int_\Omega V_2\,u_2^2\,dx\right|\le C\,\|f\|_{W^{-1,2}(\Omega)}\,\|u_1-u_2\|_{W^{1,2}_0(\Omega)},
\end{equation}
where 
\[
C=1+\frac{T_N}{T_N-\|(V_1)_-\|_{L^{N/2}(\Omega)}}+\frac{T_N}{T_N-\|(V_2)_-\|_{L^{N/2}(\Omega)}}.
\]
In particular $C=3$ if $V_1,V_2$ are nonnegative.
\end{lm}

\begin{proof}
From the respective PDEs, we obtain
\[
\int_\Omega|\nabla u_i|^2\,dx+\int_\Omega V_i\,u^2_i\,dx=\langle f,u_i\rangle,\qquad i=1,2,
\]
so that
\[
\left|\int_\Omega V_1\,u_1^2\,dx-\int_\Omega V_2\,u_2^2\,dx\right|\le\left|\int_\Omega|\nabla u_1|^2\,dx-\int_\Omega|\nabla u_2|^2\,dx\right|+\|f\|_{W^{-1,2}(\Omega)}\|u_1-u_2\|_{W^{1,2}_0(\Omega)}.
\]
Finally, we observe that 
\[
\begin{split}
\left|\|\nabla u_1\|^2_{L^2(\Omega)}-\|\nabla u_2\|^2_{L^2(\Omega)}\right|
&=\Big(\|\nabla u_1\|_{L^2(\Omega)}+\|\nabla u_2\|_{L^2(\Omega)}\Big)\,\left|\|\nabla u_1\|_{L^2(\Omega)}-\|\nabla u_2\|_{L^2(\Omega)}\right|\\
&\le\Big(\|\nabla u_1\|_{L^2(\Omega)}+\|\nabla u_2\|_{L^2(\Omega)}\Big)\,\|u_1-u_2\|_{W^{1,2}_0(\Omega)},
\end{split}
\]
then we can conclude by using \eqref{energyestimate}.
\end{proof}

\begin{lm}
Let $V_1,V_2\in\mathcal{V}$ be two admissible potentials. Let $u_1,u_2\in W^{1,2}_0(\Omega)$ be solutions of \eqref{problema} such that 
\[
u_i\in L^2(\Omega;V_1)\cap L^2(\Omega;V_2),\qquad i=1,2.
\]
Then we have
\begin{equation}\label{strano}
\mathcal{E}_f(V_1)-\mathcal{E}_f(V_2)=\frac{1}{2}\int_\Omega (V_1-V_2)\,u_1\,u_2\,dx.
\end{equation}
In particular there holds
\begin{equation}\label{strano2}
|\mathcal{E}_f(V_1)-\mathcal{E}_f(V_2)|\le\frac{1}{2}\left(\int_\Omega|V_1-V_2|\,u_1^2\,dx\right)^{1/2}\left(\int_\Omega|V_1-V_2|\,u_2^2\,dx\right)^{1/2}.
\end{equation}
\end{lm}

\begin{proof}
We first observe that, from the hypothesis on the potentials, we can use $u_1-u_2$ as a test function for the equations solved by $u_1$ and $u_2$, i.e.
\begin{equation}\label{equazioni}
\int_\Omega\langle\nabla u_i,\nabla\varphi\rangle\,dx+\int_\Omega V_i\,u_i\,\varphi\,dx=\langle f,\varphi\rangle,\qquad \mbox{ for every } \varphi\in W^{1,2}_0(\Omega)\cap L^2(\Omega;V_i),\ i=1,2.
\end{equation}
We have
\[
\begin{split}
\mathcal{E}_f(V_1)-\mathcal{E}_f(V_2)&=\frac{1}{2}\int_\Omega|\nabla u_1|^2\,dx+\frac{1}{2}\int_\Omega V_1\,u^2_1\,dx-\langle f,u_1\rangle\\
&-\frac{1}{2}\int_\Omega|\nabla u_2|^2\,dx-\frac{1}{2}\int_\Omega V_2\,u^2_2\,dx+\langle f,u_2\rangle.
\end{split}
\]
On the other hand
\[
\begin{split}
\frac{1}{2}\int_\Omega|\nabla u_1|^2\,dx-\frac{1}{2}\int_\Omega|\nabla u_2|^2\,dx
&=\frac{1}{2}\int_\Omega\langle\nabla u_1,\nabla(u_1-u_2)\rangle\,dx\\
&-\frac{1}{2}\int_\Omega\langle\nabla u_2,\nabla(u_2-u_1)\rangle\,dx,
\end{split}
\]
thus by appealing to \eqref{equazioni}, we get
\[
\begin{split}
\frac{1}{2}\int_\Omega|\nabla u_1|^2\,dx-\frac{1}{2}\int_\Omega|\nabla u_2|^2\,dx
&=-\frac{1}{2}\int_\Omega V_1\,u_1\,(u_1-u_2)\,dx+\frac{1}{2}\,\langle f,u_1-u_2\rangle\\
&+\frac{1}{2}\int_\Omega V_2\,u_2\,(u_2-u_1)\,dx-\frac{1}{2}\,\langle f,u_2-u_1\rangle\\
&=-\frac{1}{2}\int_\Omega V_1\,u_1^2\,dx+\frac{1}{2}\int_\Omega V_2\,u_2^2\,dx+\langle f,u_1-u_2\rangle\\
&+\frac{1}{2}\int_\Omega(V_1-V_2)\,u_1\,u_2\,dx.
\end{split}
\]
This concludes the proof of \eqref{strano}. The estimate \eqref{strano2} just follows by applying H\"older inequality.
\end{proof}

\begin{oss}
\label{oss:lipschitz}
Observe that if $V_1,V_2\in L^p(\Omega)$ with $p>1$ and $u_1,u_2\in L^{2p/(p-1)}(\Omega)$, then the hypotheses of the previous Lemma are verified and \eqref{strano2} gives the Lipschitz estimate
\[
|\mathcal{E}_f(V_1)-\mathcal{E}_f(V_2)|\le\frac{1}{2}\,\|V_1-V_2\|_{L^p(\Omega)}\,\sqrt{\|u_1\|_{L^{2p/(p-1)}(\Omega)}}\,\sqrt{\|u_2\|_{L^{2p/(p-1)}(\Omega)}}.
\]
By Proposition \ref{prop:embed}, the condition $u_1,u_2\in L^{2p/(p-1)}(\Omega)$ is verified for example if
\begin{itemize}
\item $|\Omega|=+\infty$ and $p=N/2$;
\vskip.2cm
\item $\Omega=\omega\times\mathbb{R}$ is a waveguide and $p\ge N/2$;
\vskip.2cm
\item $|\Omega|<+\infty$ and $p\ge N/2$.
\end{itemize}
\end{oss}

%%%%%%%%%%%%%%%%%%%%%%%%%%%%%%%%%%%%%%%%%%%%%%%%%%
\section{Maximization problems}\label{sec:max}

In this section we fix $p>1$ and we consider the optimization problem for potentials
\begin{equation}\label{maximization}
\max_{V\in\mathcal{V}}\left\{\mathcal{E}_f(V)\ :\ \int_\Omega|V|^p\,dx\le 1\right\}.
\end{equation}
We also introduce the strictly convex functional
\begin{equation}\label{G}
G_{p,f}(u)=\frac{1}{2}\int_\Omega|\nabla u|^2\,dx+\frac{1}{2}\left(\int_\Omega|u|^{2p/(p-1)}\,dx\right)^{(p-1)/p}-\langle f,u\rangle,\qquad u\in W^{1,2}_0(\Omega),
\end{equation}
where it is intended that $G_{p,f}(u)=+\infty$ if $u\not\in L^{2p/(p-1)}(\Omega)$. We recall the following existence result from \cite{BGRV}. We give the proof for the reader's convenience.

\begin{prop}
The problem \eqref{maximization} admits a solution and is equivalent to
\begin{equation}
\label{maximizationbis}
\max_{V\in\mathcal{V}}\left\{\mathcal{E}_f(V)\ :\ V\ge0,\ \int_\Omega V^p\,dx=1\right\}.
\end{equation}
The solution $V_0$ is unique and is of the form
\begin{equation}\label{relazione}
V_0=\left(\int_\Omega|v_0|^{2\,p/(p-1)}\,dx\right)^{-1/p}v_0^{2/(p-1)},
\end{equation}
where $v_0\in W^{1,2}_0(\Omega)\cap L^{2p/(p-1)}(\Omega)$ is the unique minimizer of $G_{p,f}$. Moreover, we have
\begin{equation}\label{gagliardo}
\mathcal{E}_{f}(V_0)=G_{p,f}(v_0).
\end{equation}
\end{prop}

\begin{proof}
We start by proving that we can restrict the optimization to positive potentials that saturate the constraint on the $L^p$ norm. We have
\[
\sup_{V\in\mathcal{V}}\left\{\mathcal{E}_f(V)\ :\ V\ge0,\ \int_\Omega|V|^p\,dx=1\right\}\le\sup_{V\in\mathcal{V}}\left\{\mathcal{E}_f(V)\ :\ \int_\Omega|V|^p\,dx\le 1\right\}.
\]
On the other hand it is immediate to see that
\[
\mathcal{E}_f(V)\le\mathcal{E}_f\left(\frac{|V|}{\|V\|_{L^p(\Omega)}}\right),\qquad\mbox{ for every }V\in L^p(\Omega)\setminus\{0\} \mbox{ with } \int_\Omega |V|^p\, dx\le 1,
\]
thus the two suprema coincide.
\vskip.2cm\noindent 
In order to characterize the optimal potential $V_0$, we observe that for every $u\in L^{2\,p/(p-1)}(\Omega)$ and every admissible potential, we get
\begin{equation}
\label{holders}
\int_\Omega V\,u^2\,dx\le\left(\int_\Omega |u|^{2\,p/(p-1)}\,dx\right)^{(p-1)/p},
\end{equation}
thanks to H\"older inequality. By appealing to the definition of the energy $\mathcal{E}_f(V)$, we then get
\[
\mathcal{E}_{f}(V)\le \frac{1}{2}\,\int_\Omega|\nabla u|^2\,dx+\frac{1}{2}\,\left(\int_\Omega|u|^{2p/(p-1)}\,dx\right)^{(p-1)/p}-\langle f,u\rangle,\qquad\forall u\in W^{1,2}_0(\Omega).
\]
By taking the infimum on $u$, we obtain
\[
\mathcal{E}_f(V)\le\min_{u\in W^{1,2}_0(\Omega)}G_{p,f}(u),\qquad\mbox{ for every }V\mbox{ admissible}.
\]
On the other hand, we see that if $v_0$ is a minimizer of $G_{p,f}$ and $(V_0,v_0)$ achieves equality in \eqref{holders}, we have equality in the last inequality. By appealing to the equality cases in H\"older inequality, we get the characterization \eqref{relazione}.
\end{proof}

\begin{oss}\label{oss:poincare}
For the sake of completeness we observe that by a standard homogeneity argument
\[
\mathcal{E}_f(V)=-\frac{1}{2}\,\sup_{u\in W^{1,2}_0(\Omega)\setminus\{0\}}\frac{\langle f,u\rangle^2}{\displaystyle\int_\Omega|\nabla u|^2\,dx+\int_\Omega Vu^2\,dx}.
\]
By using \eqref{holders} we can infer
\[
\mathcal{E}_f(V)\le-\frac{1}{2}\,\sup_{u\in W^{1,2}_0(\Omega)\setminus\{0\}}\frac{\langle f,u\rangle^2}{\displaystyle\int_\Omega|\nabla u|^2\,dx+\left(\int_\Omega |u|^{2p/(p-1)}\,dx\right)^{(p-1)/p}},
\]
for every $V\in L^p(\Omega)$ with unit norm. This implies that
\[
\sup\left\{\mathcal{E}_f(V)\ :\ \int_\Omega |V|^p\,dx=1\right\}=-\frac{1}{2}\,\sup_{u\in W^{1,2}_0(\Omega)\setminus\{0\}}\frac{\langle f,u\rangle^2}{\displaystyle\int_\Omega|\nabla u|^2\,dx+\left(\int_\Omega|u|^{2p/(p-1)}\right)^{(p-1)/p}},
\]
so that $\mathcal{E}_f(V_0)$ is related to the best constant in a Poincar\'e-Sobolev type inequality, i.e.
\[
\int_\Omega|\nabla u|^2\,dx+\left(\int_\Omega|u|^{2p/(p-1)}\,dx\right)^{(p-1)/p}\ge\frac{1}{2\,|\mathcal{E}_f(V_0)|}\,\langle u,f\rangle^2,\qquad\forall u\in W^{1,2}_0(\Omega),
\]
with equality holding if and only if $u$ is proportional to $v_0$.
\end{oss}

%%%%%%%%%%%%%%%%%%%%%%%%%%%%%%%%%%%%%%%%%%%%%%%%%%
\section{Stability for maximization problems}
\label{sec:stabmax}

In what follows $c_1$ will denote the constant
\[
c_1:=\left(\int_\Omega|v_0|^{2p/(p-1)}\,dx\right)^{(p-1)/2p},
\]
where $v_0$ is the unique minimizer of $G_{p,f}$.
In this section we prove a quantitative improvement of the inequality
\[
\mathcal{E}_f(V_0)\ge\mathcal{E}_f(V),\qquad\mbox{ for every }V\in \mathcal{V}\mbox{ such that }\int_\Omega|V|^p\,dx\le 1.
\]
At this aim, we need the following result, see \cite[Theorem 3.1]{CFL} for a proof. An earlier related result could be found in \cite[Proposition 2.6]{CEFT}.
\begin{lm}[Quantitative H\"older inequality]\label{quahol}
Let $2\le q<\infty$ and $q'=q/(q-1)$. For every $f\in L^q(\Omega)$ and $g\in L^{q'}(\Omega)$ such that $\|f\|_{L^q(\Omega)}=\|g\|_{L^{q'}(\Omega)}=1$, we have
\begin{equation}
\label{holderquant1}
\left|\int_\Omega f\,g\,dx\right|\le1-\frac{q'-1}{4}\,\Big\||f|^{q-2}f-g\Big\|^2_{L^{q'}(\Omega)},
\end{equation}
and
\begin{equation}
\label{holderquant2}
\left|\int_\Omega f\,g\,dx\right|\le 1-\frac{1}{q\,2^{q-1}}\,\Big\|f-|g|^{q'-2}g\Big\|^q_{L^q(\Omega)}.
\end{equation}
\end{lm}
\begin{oss}
In the case $q=\infty$ no quantitative inequality of the previous kind may
hold. In fact, by taking $\Omega=(0,1)$ and the functions
\[
f_n=1_{[0,1-1/n]}\qquad \mbox{ and }\qquad g=\frac{1}{2\,\sqrt{x}},
\] 
we obtain
\[
\lim_{n\to\infty} \left(\left|\int_\Omega f_n\, g\, dx\right|-1\right)=0\qquad \mbox{ while }\qquad \left\|f_n-\frac{g}{|g|}\right\|_{L^\infty(\Omega)}=\|f_n-1\|_{L^\infty(\Omega)}=1
\]
\end{oss}

\subsection{Stability of the potentials}
This is the main result of this section.
\begin{teo}[Stability of maximal potentials]\label{teo:max_stab}
Let $V_0$ be the optimal potential achieving the maximum in \eqref{maximization}. Then for every $V\in\mathcal{V}$ such that $\|V\|_{L^p(\Omega)}\le 1$ we have
\begin{equation}\label{stab}
\mathcal{E}_f(V_0)-\mathcal{E}_f(V)\ge\sigma'_M\, \left\||V|^{p-2}\, V-V_0^{p-1}\right\|^2_{L^{p'}(\Omega)},\qquad \mbox{ for } p\ge 2,
\end{equation}
and
\begin{equation}\label{stab2}
\mathcal{E}_f(V_0)-\mathcal{E}_f(V)\ge\sigma''_M\,\left\|V-V_0\right\|^2_{L^{p}(\Omega)},\qquad \mbox{ for } 1<p<2,
\end{equation}
where $\sigma_M'>0$ and $\sigma_M''>0$ are two constants depending only on $p$ and $c_1$ (see Remark \ref{oss:costanti!} below).
\end{teo}

\begin{proof}
We start observing that by hypothesis 
\[
\|V-V_0\|_{L^p(\Omega)}\le 2,\qquad \mbox{ and }\qquad \left\||V|^{p-2}\,V-|V_0|^{p-2}\, V_0\right\|_{L^{p'}(\Omega)}\le 2,
\]
thus we can always suppose
\begin{equation}
\label{unmezzo}
\mathcal{E}_f(V_0)-\mathcal{E}_f(V)\le \min\left\{\frac{c_1^2}{4},1\right\},
\end{equation}
because otherwise \eqref{stab} and \eqref{stab2} are trivially true, with constants
\[
\sigma'_M=\sigma_M''=\frac{1}{4}\, \min\left\{\frac{c_1^2}{4},1\right\}.
\]
By using $v_0$ as a test function in the variational problem defining $\mathcal{E}_f(V)$ and recalling the definition \eqref{G} of $G_{p,f}$, we get
\[
\begin{split}
\mathcal{E}_f(V)
&\le G_{p,f}(v_0)+\frac{1}{2}\left[\int_\Omega V\,v_0^2\,dx-\left(\int_\Omega|v_0|^{2p/(p-1)}\,dx\right)^{(p-1)/p}\right]\\
&=\mathcal{E}_f(V_0)+\frac{1}{2}\left[\int_\Omega V\,v_0^2\,dx-\left(\int_\Omega|v_0|^{2p/(p-1)}\,dx\right)^{(p-1)/p}\right].
\end{split}
\]
The optimal potential $V_0$ and $v_0$ are linked through \eqref{relazione}, thus by substituting above we get
\[
\begin{split}
\mathcal{E}_f(V)-\mathcal{E}_f(V_0)
&\le\frac{c_1^2}{2}\int_\Omega(V-V_0)\,V_0^{p-1}\,dx=\frac{c_1^2}{2}\,\left[\int_\Omega V\,V_0^{p-1}\,dx-1\right].
\end{split}
\]
From the previous inequality we obtain
\begin{equation}
\label{tutto}
\mathcal{E}_f(V_0)-\mathcal{E}_f(V)\ge \frac{c_1^2}{2}\, \left[1-\left(\int_\Omega |V|^p\, dx\right)^{1/p}\right]+\frac{c_1^2}{2}\,\left[\left(\int_\Omega |V|^p\, dx\right)^{1/p}-\int_\Omega V\, V_0^{p-1}\, dx\right].
\end{equation}
The two terms inside the square brackets are both positive, since $\|V\|_{L^p(\Omega)}\le\|V^{p-1}_0\|_{L^{p'}(\Omega)}=1$. The previous estimate in particular implies that
\begin{equation}
\label{scassamento}
\left(\int_\Omega |V|^p\, dx\right)^{1/p}\ge \frac{1}{2},
\end{equation}
since otherwise we would contradict \eqref{unmezzo}. We now distinguish two cases. 
\vskip.2cm\noindent
{\it Case $p\ge 2$}. By applying \eqref{holderquant1} with the choices
\[
f=\frac{V}{\|V\|_{L^p(\Omega)}},\qquad g=V_0^{p-1},\qquad q=p\quad \mbox{ and }\quad q'=p', 
\] 
we get
\begin{equation}
\label{duepezzoni}
\begin{split}
\left(\int_\Omega |V|^p\, dx\right)^{1/p}&-\int_\Omega V\, V_0^{p-1}\, dx\ge \frac{p'-1}{8}\, \left\|\frac{|V|^{p-2}\, V}{\|V\|^{p-1}_{L^p(\Omega)}}-V_0^{p-1}\right\|^2_{L^{p'}(\Omega)},
\end{split}
\end{equation}
where we used \eqref{scassamento} to estimate the $L^p$ norm of $V$ from below.
We now observe that by the triangle inequality and convexity of $t\mapsto t^{2}$, we get
\[
\begin{split}
\left\|\frac{|V|^{p-2}\, V}{\|V\|^{p-1}_{L^p(\Omega)}}-V_0^{p-1}\right\|^2_{L^{p'}(\Omega)}&\ge \frac{1}{2}\,\left\||V|^{p-2}\, V-V_0^{p-1}\right\|^2_{L^{p'}(\Omega)}-\left\|\frac{|V|^{p-2}\,V}{\|V\|^{p-1}_{L^p(\Omega)}}-|V|^{p-2}\, V\right\|^2_{L^{p'}(\Omega)}\\
&=\frac{1}{2}\, \left\||V|^{p-2}\, V-V_0^{p-1}\right\|^2_{L^{p'}(\Omega)}-\left|1-\|V\|^{p-1}_{L^p(\Omega)}\right|^2\\
&\ge \frac{1}{2}\, \left\||V|^{p-2}\, V-V_0^{p-1}\right\|^2_{L^{p'}(\Omega)}-(p-1)^2\,\left(\frac{2}{c_1^2}\right)^2\, \Big(\mathcal{E}_f(V_0)-\mathcal{E}_f(V_0)\Big)^2,
\end{split}
\]
where we used that for $p\ge 2$ 
\[
1-t^{p-1}\le (p-1)\,(1-t),\qquad \mbox{ for every } 0\le t\le 1,
\]
and \eqref{tutto} in the last inequality. By inserting this information in \eqref{duepezzoni}, combining with \eqref{tutto} and using that $\mathcal{E}_f(V_0)-\mathcal{E}_f(V)\le 1$, we end up with \eqref{stab}. 
\vskip.2cm\noindent
{\it Case $1<p<2$.} By applying \eqref{holderquant1} this time with the choices
\[
f=V_0^{p-1},\qquad g=\frac{V}{\|V\|_{L^p(\Omega)}},\qquad q=p'\quad \mbox{ and }\quad q'=p, 
\] 
we get
\begin{equation}
\label{duepezzonibis}
\begin{split}
\left(\int_\Omega |V|^p\, dx\right)^{1/p}&-\int_\Omega V\, V_0^{p-1}\, dx\ge \frac{p-1}{8}\, \left\|\frac{V}{\|V\|_{L^p(\Omega)}}-V_0\right\|^2_{L^p(\Omega)},
\end{split}
\end{equation}
where we used again \eqref{scassamento}. We can estimate the remainder term as before
\[
\begin{split}
\left\|\frac{V}{\|V\|_{L^p(\Omega)}}-V_0\right\|^2_{L^p(\Omega)}&\ge \frac{1}{2}\,\left\|V-V_0\right\|^2_{L^{p}(\Omega)}-\left\|\frac{V}{\|V\|_{L^p(\Omega)}}- V\right\|^2_{L^{p}(\Omega)}\\
&=\frac{1}{2}\,\left\|V-V_0\right\|^2_{L^{p}(\Omega)}-\left|1-\|V\|_{L^p(\Omega)}\right|^2\\
&\ge\frac{1}{2}\,\left\|V-V_0\right\|^2_{L^{p}(\Omega)}-\left(\frac{2}{c_1^2}\right)^2\,\Big(\mathcal{E}_f(V_0)-\mathcal{E}_f(V)\Big)^2,
\end{split}
\]
where we used again \eqref{tutto} in the last inequality.
We can obtain the desired result by combining \eqref{tutto}, \eqref{duepezzonibis} and the previous estimate.
\end{proof}
\begin{oss}
\label{oss:costanti!}
From the proof, we can see that a possible value for $\sigma_M'$ is 
\[
\sigma'_M=\frac{1}{4}\,\min\left\{(p'-1)\,\frac{c_1^{4}}{8\,c_1^{2}+2\,(p-1)},\,1,\,\frac{c_1^2}{4}\right\},\qquad p\ge 2,
\]
while for $\sigma''_M$ we could take
\[
\sigma''_M=\frac{1}{4}\,\min\left\{(p-1)\,\frac{c_1^4}{8\,c_1^2+2\,(p-1)},\,1,\,\frac{c_1^2}{4}\right\},\qquad 1<p<2.
\]
\end{oss}
\begin{oss}
We point out that we could have used \eqref{holderquant2} in place of \eqref{holderquant1}. In this way, one could obtain stability estimates of the type
\begin{equation}
\label{stab'}
\mathcal{E}_f(V_0)-\mathcal{E}_f(V)\ge\widetilde\sigma'_M\, \left\|V-V_0\right\|^p_{L^{p}(\Omega)},\qquad \mbox{ for } p\ge 2,
\end{equation}
and
\[
\mathcal{E}_f(V_0)-\mathcal{E}_f(V)\ge\widetilde\sigma''_M\, \left\||V|^{p-2}\, V-V_0^{p-1}\right\|^{p'}_{L^{p'}(\Omega)},\qquad \mbox{ for } 1<p< 2.
\]
We also notice that {\it these estimates are asymptotically worse} than \eqref{stab} and \eqref{stab2}. Indeed, when $V$ is of the form $V_\varepsilon=V_0+\varepsilon\,\psi$, for $\varepsilon\ll1$ and $\psi\in L^p(\Omega)$, it is not difficult to see that
\[
\|V_\varepsilon-V_0\|_{L^p(\Omega)}\simeq \varepsilon\qquad \mbox{ and }\qquad \left\||V_\varepsilon|^{p-2}\, V_\varepsilon-V_0^{p-1}\right\|_{L^{p'}(\Omega)}\simeq\varepsilon,
\]
so that we have
\[
\begin{array}{rclc}
\big\||V_\varepsilon|^{p-2}\,V_\varepsilon-V_0^{p-1}\big\|^2_{L^{p'}(\Omega)}&\gg&\|V_\varepsilon-V_0\|^p_{L^p(\Omega)}&\mbox{ if }p\ge 2,\\
\\
 \|V_\varepsilon-V_0\|^2_{L^p(\Omega)}&\gg& \big\||V_\varepsilon|^{p-2}\, V_\varepsilon-V_0^{p-1}\big\|^{p'}_{L^{p'}(\Omega)}&\mbox{ if }1<p<2.
\end{array}
\]
\end{oss}

\subsection{Stability of the state functions}
We have the following stability result for the minimization of $G_{p,f}$.
\begin{prop}\label{lm:cazzatina1}
Let $1<p<\infty$ and let $v_0$ be the unique minimizer of $G_{p,f}$ defined by \eqref{G}, then for every $u\in W^{1,2}_0(\Omega)$ we have
\begin{equation}
\label{cazzatina1}
G_{p,f}(u)-G_{p,f}(v_0)\ge \frac{1}{2}\,\|u-v_0\|^2_{W^{1,2}_0(\Omega)}.
\end{equation}
\end{prop}

\begin{proof}
We first observe that if $u\not\in L^{2p/(p-1)}(\Omega)$, then $G_{p,f}(u)=+\infty$ and \eqref{cazzatina1} trivially holds. Thus, let us take $u\in W^{1,2}_0(\Omega)\cap L^{2p/(p-1)}(\Omega)$ with $u\not=v_0$ and set
\[
d=\|u-v_0\|_{W^{1,2}_0(\Omega)}\qquad \mbox{ and }\qquad \varphi=\frac{u-v_0}{d}.
\]
Then $u$ can be written as $u=v_0+d\,\varphi$ and $\varphi$ has unitary norm in $W^{1,2}_0(\Omega)$. We then get
\[
\begin{split}
G_{p,f}(u)-G_{p,f}(v_0)&\ge d\left[\int_\Omega\langle\nabla v_0,\nabla \varphi\rangle\,dx+\left(\int_\Omega|v_0|^{2p/(p-1)}\,dx\right)^{-1/p}\int_\Omega|v_0|^{2/(p-1)}\,v_0\,\varphi\,dx-\langle f,\varphi\rangle\right]\\
&+\frac{d^2}{2}\int_\Omega|\nabla\varphi|^2\,dx,
\end{split}
\]
where we used the convexity of the $C^1$ map
\[
\Phi(s)=\|v_0+s\,\varphi\|^2_{L^{2p/(p-1)}(\Omega)}=\left(\int_\Omega |v_0+s\,\varphi|^{2p/(p-1)}\,dx\right)^{(p-1)/p},\qquad s\in\mathbb{R},
\]
so that
\[
\Phi(d)\ge\Phi(0)+\Phi'(0)d.
\]
It is now sufficient to observe that 
\[
\int_\Omega\langle\nabla v_0,\nabla\varphi\rangle\,dx+\left(\int_\Omega |v_0|^{2p/(p-1)}\,dx\right)^{-1/p}\int_\Omega|v_0|^{2/(p-1)}\,v_0\,\varphi\,dx-\langle f,\varphi\rangle=0,
\]
by minimality of $v_0$, thus we directly get \eqref{cazzatina}.
\end{proof}
As a consequence of the previous result, we get that if $(u,V)$ is almost realizing the equality in \eqref{holders}, then $u$ is near to the optimizer $v_0$ in the $W^{1,2}_0$ norm.
\begin{coro}
Let $V$ be an admissible potential for \eqref{maximization} and $u$ a corresponding energy function. Then
\begin{equation}\label{stab_weak}
\|u-v_0\|^2_{W^{1,2}_0(\Omega)}\le\left[\left(\int_\Omega|u|^{2p/(p-1)}\,dx\right)^{(p-1)/p}-\int_\Omega V\,u^2\,dx\right].
\end{equation}
\end{coro}

\begin{proof}
We already observed that
\[
\mathcal{E}_f(V)\le G_{p,f}(v),\qquad\mbox{ for every }v\in W^{1,2}_0(\Omega).
\]
By taking the energy function $v_0$ corresponding to $V_0$, we get
\[
\mathcal{E}_{f}(V)\le G_{p,f}(v_0)\le G_{p,f}(u)-\frac{1}{2}\|u-v_0\|^2_{W^{1,2}_0(\Omega)},
\]
where we used \eqref{cazzatina1}. Thus we have
\[
\frac{1}{2}\|u-v_0\|^2_{W^{1,2}_0(\Omega)}\le G_{p,f}(u)-\mathcal{E}_f(V)=\frac{1}{2}\left[\left(\int_\Omega|u|^{2p/(p-1)}\,dx\right)^{(p-1)/p}-\int_\Omega Vu^2\,dx\right],
\]
which concludes the proof.
\end{proof}

In general, for an admissible potential $V$ the corresponding energy function is not in $L^{2p/(p-1)}(\Omega)$. When this is the case, we can infer stability of the energy functions as well.
\begin{prop}
Let $V$ be admissible in \eqref{maximization}. If a corresponding energy function $u$ belongs to $L^{2p/(p-1)}(\Omega)$, then
\begin{equation}\label{functions_max}
\Big(\mathcal{E}_f(V_0)-\mathcal{E}_f(V)\Big)^\frac{1}{\vartheta(p)}\,\|u-v_0\|_{L^{2p/(p-1)}(\Omega)}\ge c\,\|u-v_0\|^{2}_{W^{1,2}_0(\Omega)},
\end{equation}
where $\vartheta(p)=\max\{2,\,p\}$ and $c>0$ is a constant depending only on $c_1,\,p$ and $\|V_-\|_{L^{N/2}(\Omega)}$.
\end{prop}

\begin{proof}
We first observe that since $u\in L^{2p/(p-1)}(\Omega)$ and $u_0\in L^{2\,p/(p-1)}(\Omega)$ as well, we have 
\[
\psi:=u-v_0\in W^{1,2}_0(\Omega)\cap L^2(\Omega;V).
\]
The function $\psi$ verifies
\[
\int_\Omega\langle\nabla\psi,\nabla\varphi\rangle\,dx+\int_\Omega (V-V_0)\,v_0\,\varphi\,dx+\int_\Omega V\,\psi\,\varphi\,dx=0,
\]
for every $\varphi\in W^{1,2}_0(\Omega)$. By using $\psi$ itself as a test function, we get
\[
\begin{split}
\int_\Omega|\nabla\psi|^2\,dx+\int_\Omega V\,\psi^2\,dx&\le\int_\Omega|V-V_0|\,|v_0|\,|\psi|\,dx\\
&\le\left(\int_\Omega|V-V_0|^p\,dx\right)^{1/p}\left(\int_\Omega \big(|v_0|\,|\psi|\big)^{p/(p-1)}\,dx\right)^{(p-1)/p}\\
&\le\|V-V_0\|_{L^p(\Omega)}\,\|v_0\|_{L^{2p/(p-1)}(\Omega)}\,\|\psi\|_{L^{2p/(p-1)}(\Omega)}.
\end{split}
\]
By noticing that from \eqref{ultimora}
\[
\int_\Omega|\nabla\psi|^2\,dx+\int_\Omega V\psi^2\,dx\ge \left(1-\frac{\|V_-\|_{L^{N/2}(\Omega)}}{T_N}\right)\,\int_\Omega |\nabla \psi|^2\, dx,
\]
and recalling that $\psi=u-v_0$, we get
\[
\left(1-\frac{\|V_-\|_{L^{N/2}(\Omega)}}{T_N}\right)\,\|u-v_0\|^2_{W^{1,2}_0(\Omega)}\le c_1\,\|V-V_0\|_{L^p(\Omega)}\, \|u-v_0\|_{L^{2p/(p-1)}(\Omega)}.
\]
Appealing to \eqref{stab2} (for $1<p<2$) or to \eqref{stab'} (for $p\ge 2$), we then get the conclusion.
\end{proof}

\begin{oss}\label{remstab}
Observe that for $p=N/2$, we have $2p/(p-1)=2^*$ and \eqref{functions_max} simply becomes
\[
\Big(\mathcal{E}_f(V_0)-\mathcal{E}_f(V)\Big)^\frac{1}{\vartheta(p)}\ge c\,\|u-v_0\|_{W^{1,2}_0(\Omega)},
\]
by Sobolev inequality, possibly with a different constant $c>0$. When $\Omega$ has a finite measure or is a waveguide and $p\ge N/2$, by Proposition \ref{prop:embed} we can always assure that the energy function $u$ belongs to $L^{2p/(p-1)}(\Omega)$ and thus we have a similar stability estimate in these cases as well.
\end{oss}

%%%%%%%%%%%%%%%%%%%%%%%%%%%%%%%%%%%%%%%%%%%%%%%%%%
\section{Minimization problems}
\label{sec:min}

In this section we consider, for a fixed $p>0$, the minimization problem
\begin{equation}\label{minimization}
\inf\left\{\mathcal{E}_f(V)\ :\ V\ge0,\ \int_\Omega \frac{1}{V^p}\,dx\le 1\right\}.
\end{equation}
\begin{oss}
Observe that this time, it is not clear whether the minimization problem on $\mathcal{V}$ without sign hypothesis, i.e.
\[
\inf_{V\in\mathcal{V}}\left\{\mathcal{E}_f(V)\ :\ \int_\Omega \frac{1}{|V|^p}\,dx\le 1\right\}.
\]
is well-posed or not, since it may happen that no admissible $V\le 0$ exist (for example if $\Omega$ is unbounded). If an optimal potential $W$ exists for the previous problem, this should be such that $W_+\not\equiv0$ and $W_-\not \equiv 0$. This seems to be an interesting issue, which we leave for future research. 
\end{oss}
We collect a couple of technical results which are needed in the sequel.
\begin{lm}
Let $V\ge0$ be such that $V^{-1}\in L^p(\Omega)$. Then we have
\begin{equation}\label{stimaLV}
\left(\int_\Omega|u|^{2p/(p+1)}\,dx\right)^{(p+1)/p}\le\left(\int_\Omega Vu^2\,dx\right)\left(\int_\Omega\frac{1}{V^p}\,dx\right)^{1/p},\qquad \mbox{ for every }u\in C^\infty_0(\Omega).
\end{equation}
In particular we have the continuous embedding $L^2(\Omega;V)\subset L^{2p/(p+1)}(\Omega)$.
\end{lm}

\begin{proof}
If $u\not\in L^2(\Omega;V)$, there is nothing to prove. So let us assume that the first integral in the right-hand side of \eqref{stimaLV} is finite. By H\"older inequality with exponents $q={(p+1)/p}$ and $q'=p+1$ we have
\[
\begin{split}
\int_\Omega|u|^{2p/(p+1)}\,dx
&=\int_\Omega|u|^{2p/(p+1)}\frac{V^{p/(p+1)}}{V^{p/(p+1)}}\,dx\le\left(\int_\Omega Vu^2\,dx\right)^{p/(p+1)}\left(\int_\Omega\frac{1}{V^p}\,dx\right)^{1/(p+1)},
\end{split}
\]
which concludes the proof.
\end{proof}

\begin{oss}\label{oss:dentro}
By standard interpolation in Lebesgue spaces, under the hypothesis of the previous result we have the continuous embedding 
\[
L^2(\Omega;V)\subset L^s(\Omega),
\]
for every $2p/(p+1)\le s\le 2^*$. In particular $L^2(\Omega;V)$ is embedded into $L^2(\Omega)$, since $2p/(p+1)$ is always strictly less than $2$. It is then not difficult to show that the operator $-\Delta +V$ has a discrete spectrum on $L^2(\Omega)$.
\end{oss}
The following energy estimate will be needed in the sequel. 
\begin{lm}
Let $f\in W^{-1,2}(\Omega)$, $V$ be an admissible potential for \eqref{minimization} and $u$ its energy function. Then we have
\begin{equation}\label{energyestimate2}
\left(\int_\Omega|u|^{2p/(p+1)}\,dx\right)^{(p+1)/p}\le\frac{1}{2}\,\|f\|^2_{W^{-1,2}(\Omega)}.
\end{equation}
\end{lm}
\begin{proof}
By using the equation and Young ienquality, we have
\[
\int_\Omega |\nabla u|^2\, dx+\int_\Omega V\, u^2\, dx=\langle f,u\rangle\le \frac{1}{2}\, \|f\|_{W^{-1,2}(\Omega)}^2+\frac{1}{2}\,\int_\Omega |\nabla u|^2\, dx.
\]
From the previous we obtain
\[
\int_\Omega V\, u^2\, dx\le \frac{1}{2}\, \|f\|_{W^{-1,2}(\Omega)}^2,
\]
then it is sufficient to apply \eqref{stimaLV}.
\end{proof}
Let $1<p<\infty$, in what follows we set for simplicity
\[
J_{p,f}(u)=\frac{1}{2}\int_\Omega|\nabla u|^2\,dx+\frac{1}{2}\left(\int_\Omega |u|^{2p/(p+1)}\,dx\right)^{(p+1)/p}-\langle f,u\rangle\qquad\forall u\in W^{1,2}_0(\Omega),
\]
where it is intended that $J_{p,f}(u)=+\infty$ if $u\not\in L^{2p/(p+1)}(\Omega)$. Again, it is not difficult to see that $J_{p,f}$ admits a unique minimizer.
We recall the following result from \cite{BGRV}.

\begin{prop}\label{prop:minimo}
The problem \eqref{minimization} admits a unique solution $U_0$ of the form
\begin{equation}\label{relazionem}
U_0=\left(\int_\Omega|u_0|^{2p/(p+1)}\,dx\right)^{1/p}|u_{0}|^{-2/(p+1)},
\end{equation}
where $u_0\in W^{1,2}_0(\Omega)\cap L^{2p/(p+1)}(\Omega)$ is the unique minimizer of $J_{p,f}$.
\end{prop}

\begin{proof}
Let $u$ be the energy function corresponding to $V$, then we have
\[
\mathcal{E}_f(V)=J_{p,f}(u)+\frac{1}{2}\left[\int_\Omega Vu^2-\left(\int_\Omega|u|^{2p/(p+1)}\right)^{(p+1)/p}\right].
\]
By using \eqref{stimaLV} and the minimimality of $u_0$, we get
\[
\mathcal{E}_f(V)\ge J_{p,f}(u_0)=\min_{u\in W^{1,2}_0(\Omega)}J_{p,f}(u).
\]
By appealing again to the equality cases in H\"older inequality, we get the characterization of the optimal potential $V_0$.
\end{proof}

\begin{oss}
As in Remark \ref{oss:poincare}, by applying \eqref{stimaLV} we can infer
\[
\mathcal{E}_f(V)\ge-\frac{1}{2}\,\sup_{u\in W^{1,2}_0(\Omega)\setminus\{0\}}\frac{\langle f,u\rangle^2}{\displaystyle\int_\Omega|\nabla u|^2\,dx+\left(\int_\Omega|u|^{2p/(p+1)}\right)^{(p+1)/p}},
\]
for every $V$ admissible, so that
\[
\inf\left\{\mathcal{E}_f(V)\ :\ V\ge0,\ \int_\Omega V^{-p}\,dx\le 1\right\}=-\frac{1}{2}\sup_{u\in W^{1,2}_0(\Omega)\setminus\{0\}}\frac{\langle f,u\rangle^2}{\displaystyle\int_\Omega|\nabla u|^2\,dx+\left(\int_\Omega|u|^{2p/(p+1)}\right)^{(p+1)/p}}.
\]
Thus we have
\[
\int_\Omega|\nabla u|^2\,dx+\left(\int_\Omega|u|^{2p/(p+1)}\,dx\right)^{(p+1)/p}\ge\frac{1}{2\,|\mathcal{E}_f(U_0)|}\langle u,f\rangle^2,\qquad u\in W^{1,2}_0(\Omega)\cap L^{2p/(p+1)}(\Omega),
\]
with equality if and only if $u$ is proportional to $u_0$.
\end{oss}

%%%%%%%%%%%%%%%%%%%%%%%%%%%%%%%%%%%%%%%%%%%%%%%%%%
\section{Stability for minimization problems}\label{sec:stabmin}

In this section we take $p>1$ and we still denote by $u_0$ the unique minimizer of $J_{p,f}$. We also set
\begin{equation}\label{c_2}
c_2:=\left(\int_\Omega|u_0|^{2p/(p+1)}\,dx\right)^{(p+1)/2p}.
\end{equation}

\subsection{Preliminary results}
We start with a stability result for the minimization of $J_{p,f}$. The proof is the same as that of Lemma \ref{lm:cazzatina1}, thus we omit it.
\begin{prop}\label{prop:cazzatina}
Let $1<p<\infty$ and let $u_0$ be the unique minimizer of $J_{p,f}$. Then for every $u\in W^{1,2}_0(\Omega)$ we have
\begin{equation}\label{cazzatina}
J_{p,f}(u)-J_{p,f}(u_0)\ge\frac{1}{2}\|u-u_0\|^2_{W^{1,2}_0(\Omega)}.
\end{equation}
\end{prop}

We also need the following result, asserting that the energy gap controls the difference of the $L^{2\,p/(p+1)}(\Omega)$ norms.
\begin{lm}
\label{lm:corner}
Let $f\in W^{-1,2}(\Omega)$ and let $V$ be an admissible potential for \eqref{minimization}. If we suppose that $\mathcal{E}_f(V)-\mathcal{E}_f(U_0)\le 1$, then
\begin{equation}\label{normaLmerda}
\sqrt{\mathcal{E}_f(V)-\mathcal{E}_f(U_0)}\ge c_3\left|\left(\int_\Omega|u_0|^{2p/(p+1)}\,dx\right)^{(p+1)/p}-\left(\int_\Omega|u|^{2p/(p+1)}\,dx\right)^{(p+1)/p}\right|,
\end{equation}
for a constant $c_3>0$ depending only on $\|f\|_{W^{-1,2}(\Omega)}$.
\end{lm}

\begin{proof}
We first observe that
\begin{equation}\label{fundamental}
\mathcal{E}_f(V)-\mathcal{E}_f(U_0)=\big[J_{p,f}(u)-J_{p,f}(u_0)\big]+\frac{1}{2}\left[\int_\Omega Vu^2\,dx-\left(\int_\Omega|u|^{2p/(p+1)}\,dx\right)^{(p+1)/p}\right],
\end{equation}
and both terms inside the square brackets are positive. In particular we get
\[
\mathcal{E}_f(V)-\mathcal{E}_f(U_0)\ge\frac{1}{2}\left[\int_\Omega Vu^2\,dx-\left(\int_\Omega|u|^{2p/(p+1)}\,dx\right)^{(p+1)/p}\right].
\]
By using the estimate on the weighted $L^2$ norms \eqref{L2V}, we then get
\begin{equation}\label{guasi}
\mathcal{E}_f(V)-\mathcal{E}_f(U_0)\ge\frac{1}{2}\left[\int_\Omega U_0\,u_0^2\,dx-\left(\int_\Omega|u|^{2p/(p+1)}\,dx\right)^{(p+1)/p}\right]-\frac32\|f\|_{W^{-1,2}(\Omega)}\|u-u_0\|_{W^{1,2}(\Omega)}.
\end{equation}
We now use that by \eqref{relazionem} $U_0$ and $u_0$ are linked through
\[
U_0=\left(\int_\Omega|u_0|^{2p/(p+1)}\,dx\right)^{1/p}|u_0|^{-2/(p+1)},
\]
so that
\begin{equation}\label{ugualeu_0}
\int_\Omega U_0\,u_0^2\,dx=\left(\int_\Omega|u_0|^{2p/(p+1)}\,dx\right)^{(p+1)/p}.
\end{equation}
If we use this in \eqref{guasi}, we end up with
\[
\begin{split}
\mathcal{E}_f(V)-\mathcal{E}_f(U_0)&+\frac{3}{2}\,\|f\|_{W^{-1,2}(\Omega)}\,\|u-u_0\|_{W^{1,2}(\Omega)}\\
&\ge\frac{1}{2}\left[\left(\int_\Omega|u_0|^{2p/(p+1)}\,dx\right)^{(p+1)/p}-\left(\int_\Omega|u|^{2p/(p+1)}\,dx\right)^{(p+1)/p}\right].
\end{split}
\]
By recalling \eqref{cazzatina}, \eqref{fundamental} and using the hypothesis $\mathcal{E}_f(V)-\mathcal{E}_f(U_0)\le1$, we then get
\[
\sqrt{\mathcal{E}_f(V)-\mathcal{E}_f(U_0)}\ge\frac{1}{C}\left[\left(\int_\Omega|u_0|^{2p/(p+1)}\,dx\right)^{(p+1)/p}-\left(\int_\Omega|u|^{2p/(p+1)}\,dx\right)^{(p+1)/p}\right].
\]
On the other hand, we have
\[
\begin{split}
\left(\int_\Omega|u|^{2p/(p+1)}\,dx\right)^{(p+1)/p}-\left(\int_\Omega|u_0|^{2p/(p+1)}\,dx\right)^{(p+1)/p}\le\int_\Omega Vu^2\,dx-\int_\Omega U_0\,u_0^2\,dx,
\end{split}
\]
where we used \eqref{stimaLV} and \eqref{ugualeu_0}. If we now apply \eqref{L2V}, we get 
\[
\left[\left(\int_\Omega|u|^{2p/(p+1)}\,dx\right)^{(p+1)/p}-\left(\int_\Omega|u_0|^{2p/(p+1)}\,dx\right)^{(p+1)/p}\right]\le 3\, \|f\|_{W^{-1,2}(\Omega)}\,\|u-u_0\|_{W^{1,2}_0(\Omega)},
\]
and again we can conclude thanks to \eqref{fundamental} and \eqref{cazzatina}.
\end{proof}

\begin{oss}
A closer inspection of the previous proof ensures that we can take
\[
c_3:=\min\left\{\frac{\sqrt{2}}{\sqrt{2}+3\,\|f\|_{W^{-1,2}(\Omega)}},\,\frac{1}{3\sqrt{2}\,\|f\|_{W^{-1,2}(\Omega)}}\right\}.
\]
\end{oss}
The following result guarantees that it is sufficient to prove stability for potentials saturating the constraint $\int_\Omega V^{-p}\le 1$.
\begin{lm}[Reduction Lemma]
\label{lm:reduction}
Let $V$ be admissible in \eqref{minimization} and such that $\int_\Omega V^{-p}\, dx<1$. Let us suppose that its energy function $u$ satisfies
\begin{equation}
\label{beta}
\int_\Omega V\, u^2\, dx\ge \beta>0.
\end{equation}
Then there exists a potential $U\ge 0$ with $\int_\Omega U^{-p}\, dx=1$ such that
\begin{equation}
\label{reduction1}
\mathcal{E}_f(V)-\mathcal{E}_f(U_0)\ge \mathcal{E}_f(U)-\mathcal{E}_f(U_0),
\end{equation} 
and
\begin{equation}
\label{reduction2}
\left\|\frac{1}{V}-\frac{1}{U_0}\right\|_{L^p(\Omega)}\le \left\|\frac{1}{U}-\frac{1}{U_0}\right\|_{L^p(\Omega)}+\frac{2}{\beta}\, \Big(\mathcal{E}_f(V)-\mathcal{E}_f(U_0)\Big).
\end{equation}
\end{lm}
\begin{proof}
Let $\lambda=\|V^{-1}\|_{L^p(\Omega)}<1$, then we set $U=\lambda\, V$. It is clear that $U<V$, so that $\mathcal{E}_f(U)\le \mathcal{E}_f(V)$ by the definition of the energy and the first property \eqref{reduction1} follows. In order to prove the second, we observe that by \eqref{fundamental}, we have
\[
\begin{split}
\mathcal{E}_f(V)-\mathcal{E}_f(U_0)&\ge \frac{1}{2}\,\left[\int_\Omega V\,u^2\,dx-\left(\int_\Omega|u|^{2p/(p+1)}\,dx\right)^{(p+1)/p}\right]\\
&\ge\frac{1}{2}\, \int_\Omega V\, u^2\, dx\,\left[1-\left(\int_\Omega V^{-p}\, dx\right)^{1/p}\right]=\frac{1}{2}\, \left(\int_\Omega V\, u^2\, dx\right)\, |1-\lambda|,
\end{split}
\]
where we also used \eqref{stimaLV} in the second inequality. By using the hypothesis on $u$ and the definition of $U$, we get
\[
|1-\lambda|\le \frac{2}{\beta}\,\Big(\mathcal{E}_f(V)-\mathcal{E}_f(U_0)\Big),
\]
and since by the triangle inequality
\[
\begin{split}
\left\|\frac{1}{V}-\frac{1}{U_0}\right\|_{L^p(\Omega)}&\le \left\|\frac{1}{U}-\frac{1}{U_0}\right\|_{L^p(\Omega)}+|1-\lambda|,
\end{split}
\]
we get the desired conclusion.
\end{proof}

Finally, the following very simple estimate will be quite useful.

\begin{lm}\label{lm:sospetto}
Let $1<r,s<\infty$ and $g,g_0\in L^r(\Omega)\cap L^s(\Omega)$. Then we have
\begin{equation}
\label{triangolono!}
\|g_0\|_{L^r(\Omega)}\le\|g\|_{L^r(\Omega)}+\|g-g_0\|_{L^s(\Omega)}\,\frac{\left\||g_0|^{r-1}\right\|_{L^{s'}(\Omega)}}{\|g_0\|^{r-1}_{L^r(\Omega)}}.
\end{equation}
\end{lm}
\begin{proof}
We can suppose that $|g_0|^{r-1}\in L^{s'}(\Omega)$, otherwise there is nothing to prove. For every $\varphi\in L^{r'}(\Omega)\cap L^{s'}(\Omega)$ we have
\[
\begin{split}
\left|\int_\Omega g_0\,\varphi\,dx\right|
&\le\int_\Omega|g-g_0|\,|\varphi|\,dx+\left|\int_\Omega g\,\varphi\,dx\right|\\
&\le \|g-g_0\|_{L^s(\Omega)}\|\varphi\|_{L^{s'}(\Omega)}+\|g\|_{L^r(\Omega)}\,\|\varphi\|_{L^{r'}(\Omega)}.
\end{split}
\]
If we now choose 
\[
\varphi=|g_0|^{r-2}\,g_0,
\]
and then simplify by $\|g_0\|^{r-1}_{L^r(\Omega)}$ on both sides, we obtain \eqref{triangolono!}.
\end{proof}

\subsection{Stability of the potentials}

The following is the main result of this section, which is proved under the integrability assumption \eqref{ipotesiu0} on $u_0$. For a discussion on this hypothesis, we refer the reader to Remark \ref{oss:ipotesi} below.
\begin{teo}[Stability of minimal potentials]
\label{teo:min_stab}
Let $U_0$ be the optimal potential achieving the minimum in \eqref{minimization}. 
%Let $N\ge3$, $1<p<\infty$ and $f\in W^{-1,2}(\Omega)\setminus\{0\}$. 
Let us suppose that the optimal function $u_0$ is such that
\begin{equation}
\label{ipotesiu0}
c_4:=\left\||u_0|^{(p-1)/(p+1)}\right\|_{L^{(2^*)'}(\Omega)}<+\infty.
\end{equation}
Then for every positive potential $V$ such that $\|1/V\|_{L^p(\Omega)}\le 1$ we have
\begin{equation}\label{eccolo}
\mathcal{E}_f(V)-\mathcal{E}_f(U_0)\ge\sigma_m\,\left\|\frac{1}{V}-\frac{1}{U_0}\right\|^{2\,p\,(p+1)/(p-1)}_{L^p(\Omega)},
\end{equation}
for a constant $\sigma_m>0$ depending only on $N,p,c_2,c_4$ and $\|f\|_{W^{-1,2}(\Omega)}$ (see Remark \ref{oss:sigma} below).
\end{teo}
\begin{proof}
We divide the proof into various steps.
\vskip.2cm\noindent
\underline{\it Reduction step.} Let $V$ be a potential admissible in \eqref{minimization}. We set
\begin{equation}
\label{c5}
c_5=\min\left\{1,\,\left(\frac{c_2^2\, c_3}{2}\right)^2\right\},
\end{equation}
where we recall that $c_2$ is the $L^{2\,p/(p+1)}(\Omega)$ norm of $u_0$ and $c_3$ is the constant in \eqref{normaLmerda}.
Since by hypothesis
\[
\left\|\frac{1}{V}-\frac{1}{U_0}\right\|_{L^p(\Omega)}\le 2,
\] 
we can always assume 
\[
\mathcal{E}_f(V)-\mathcal{E}_f(U_0)\le c_5,
\]
otherwise \eqref{eccolo} is trivially true with the constant $\sigma_m=c_5\,4^{-p(p+1)/(p-1)}$. Under this assumption, by definition of $c_5$ and Lemma \ref{lm:corner} we have that the energy function $v$ of $V$ verifies
\[
\left(\int_\Omega|v|^{2p/(p+1)}\,dx\right)^{(p+1)/p}\ge\frac{1}{2}\left(\int_\Omega|u_0|^{2p/(p+1)}\,dx\right)^{(p+1)/p}=\frac{c_2^2}{2}.
\]
This in turn implies that $v$ verifies \eqref{beta} with $\beta=c^2_2/2$ thanks to \eqref{stimaLV}. From the Reduction Lemma \ref{lm:reduction} we thus obtain that there exists a positive potential $U$ with $\|1/U\|_{L^p(\Omega)}=1$ such that
\[
\left\|\frac{1}{V}-\frac{1}{U_0}\right\|_{L^p(\Omega)}\le \left\|\frac{1}{U}-\frac{1}{U_0}\right\|_{L^p(\Omega)}+\frac{4}{c_2^2}\, \Big(\mathcal{E}_f(V)-\mathcal{E}_f(U_0)\Big)
\]
and
\[
\mathcal{E}_f(U)-\mathcal{E}_f(U_0)\le \mathcal{E}_f(V)-\mathcal{E}_f(U_0).
\]
We are going to prove the stability estimate \eqref{eccolo} for the potential $U$. Observe that since the energy gap has decreased, we still have
\begin{equation}\label{piccolo}
\mathcal{E}_f(U)-\mathcal{E}_f(U_0)\le c_5,
\end{equation}
and thus again
\begin{equation}\label{dalcanto}
\left(\int_\Omega|u|^{2p/(p+1)}\,dx\right)^{(p+1)/p}\ge\frac{c_2^2}{2}.
\end{equation}
where $u$ is now the energy function of $U$.
From \eqref{fundamental} and Proposition \ref{prop:cazzatina}, we have
\begin{equation}
\label{primotempo}
\mathcal{E}_f(U)-\mathcal{E}_f(U_0)\ge\frac{1}{2}\left(\mathcal{I}_1+\mathcal{I}_2\right),
\end{equation}
where we introduced the notation
\[
\mathcal{I}_1=\int_\Omega U\,u^2\, dx-\left(\int_\Omega |u|^{2p/(p+1)}\right)^{(p+1)/p}\qquad\mbox{ and }\qquad\mathcal{I}_2=\|u-u_0\|^2_{W^{1,2}_0(\Omega)}.
\]
We proceed to estimate $\mathcal{I}_1$ and $\mathcal{I}_2$ separately.
\vskip.2cm\noindent 
\underline{\it Estimate on $\mathcal{I}_1$.} For this we use the quantitative H\"older inequality \eqref{holderquant2} with
\[
q=p+1,\qquad q'=\frac{p+1}{p},\qquad f=U^{-p/(p+1)},\qquad g=
\frac{|u|^{2p/(p+1)}\,U^{p/(p+1)}}{\displaystyle\left(\int_\Omega U\,u^2\,dx\right)^{p/(p+1)}}.
\]
Thus we get
\[
\begin{split}
\left(\int_\Omega U\,u^2\,dx\right)^{p/(p+1)}&-\int_\Omega|u|^{2p/(p+1)}\,dx\\
&\ge\frac{1}{2^p(p+1)}\left(\int_\Omega U\,u^2\, dx\right)^{p/(p+1)}\,\left\|\frac{|u|^{2/(p+1)}\,U^{1/(p+1)}}{\displaystyle\left(\int_\Omega U\,u^2\,dx\right)^{1/(p+1)}}-U^{-p/(p+1)}\right\|^{p+1}_{L^{p+1}(\Omega)}.
\end{split}
\]
By using \eqref{stimaLV} and \eqref{dalcanto} we have
\[
\left(\int_\Omega U\,u^2\, dx\right)^\frac{p}{p+1}\ge \left(\frac{c_2^2}{2}\right)^\frac{p}{p+1},
\]
and by convexity of the function $t\mapsto t^{(p+1)/p}$, we have
\[
\int_\Omega U\, u^2\, dx-\left(\int_\Omega |u|^\frac{2\,p}{p+1}\, dx\right)^\frac{p+1}{p}\ge \frac{p+1}{p}\,\left(\int_\Omega |u|^\frac{2\,p}{p+1}\, dx\right)^\frac{1}{p}\, \left[\left(\int_\Omega U\, u^2\, dx\right)^\frac{p}{p+1}-\int_\Omega |u|^\frac{2\,p}{p+1}\ ,dx\right].
\]
Thus, for the moment we obtained
\begin{equation}
\label{boh}
\begin{split}
\mathcal{I}_1\ge \frac{c^2_2}{p\,2^{p+1}}\,\left\|\frac{|u|^\frac{2}{p+1}\, U^\frac{1}{p+1}}{\displaystyle\left(\int_\Omega u^2\, U\, dx\right)^\frac{1}{p+1}}-\frac{1}{U^\frac{p}{p+1}}\right\|^{p+1}_{L^{p+1}(\Omega)},
\end{split}
\end{equation}
where we used again \eqref{dalcanto} to estimate the norm of $u$ from below.
Observe that we have
\[
\begin{split}
\left\|\frac{|u|^\frac{2}{p+1}\, U^\frac{1}{p+1}}{\displaystyle\left(\int_\Omega u^2\, U\, dx\right)^\frac{1}{p+1}}-\frac{1}{U^\frac{p}{p+1}}\right\|^{p+1}_{L^{p+1}(\Omega)}&=\left\|\frac{|u|^\frac{2}{p+1}}{\left(\displaystyle\int_\Omega u^2\, U\, dx\right)^\frac{1}{p+1}}-\frac{1}{U}\right\|^{p+1}_{L^{p+1}(\Omega;U)}\\
&=\left\|\left|\frac{|u|^\frac{2}{p+1}}{\left(\displaystyle\int_\Omega u^2\, U\, dx\right)^\frac{1}{p+1}}-\frac{1}{U}\right|^\frac{p+1}{2}\right\|^2_{L^{2}(\Omega;U)}
\end{split}
\]
then by applying \eqref{stimaLV} we get
\[
\left\|\frac{|u|^\frac{2}{p+1}\, U^\frac{1}{p+1}}{\displaystyle\left(\int_\Omega u^2\, U\, dx\right)^\frac{1}{p+1}}-\frac{1}{U^\frac{p}{p+1}}\right\|^{p+1}_{L^{p+1}(\Omega)}\ge \left\| \frac{|u|^\frac{2}{p+1}}{\left(\displaystyle\int_\Omega u^2\, U\, dx\right)^\frac{1}{p+1}}-\frac{1}{U}\right\|^{p+1}_{L^p(\Omega)}.
\]
We now use the triangle inequality and the convexity of $t\mapsto t^{p+1}$, so that
\[
\begin{split}
\left\|\frac{|u|^\frac{2}{p+1}}{\left(\displaystyle\int_\Omega u^2\, U\, dx\right)^\frac{1}{p+1}}-\frac{1}{U}\right\|^{p+1}_{L^{p}(\Omega)}&\ge \frac{1}{2^p}\,\left\|\frac{|u|^\frac{2}{p+1}}{\left(\displaystyle\int_\Omega |u|^\frac{2\,p}{p+1}\, dx\right)^\frac{1}{p}}-\frac{1}{U}\right\|^{p+1}_{L^{p}(\Omega)}\\
&-\left\|\frac{|u|^\frac{2}{p+1}}{\left(\displaystyle\int_\Omega u^2\, U\, dx\right)^\frac{1}{p+1}}-\frac{|u|^\frac{2}{p+1}}{\left(\displaystyle\int_\Omega |u|^\frac{2\,p}{p+1}\, dx\right)^\frac{1}{p}}\right\|^{p+1}_{L^{p}(\Omega)}.
\end{split}
\]
The last term simply gives
\[
\left\|\frac{|u|^\frac{2}{p+1}}{\left(\displaystyle\int_\Omega u^2\, U\, dx\right)^\frac{1}{p+1}}-\frac{|u|^\frac{2}{p+1}}{\left(\displaystyle\int_\Omega |u|^\frac{2\,p}{p+1}\, dx\right)^\frac{1}{p}}\right\|^{p+1}_{L^{p}(\Omega)}=\frac{\left|\left(\displaystyle\int_\Omega |u|^\frac{2\,p}{p+1}\, dx\right)^\frac{1}{p}-\left(\displaystyle\int_\Omega u^2\, U\, dx\right)^\frac{1}{p+1}\right|^{p+1}}{\displaystyle\int_\Omega u^2\, U\, dx}.
\]
By keeping everything together, we have obtained
\begin{equation}
\label{boh2}
\begin{split}
\mathcal{I}_1&\ge \frac{c_2^2}{p\,2^{2\,p+1}}\,\left\|\frac{|u|^\frac{2}{p+1}}{\left(\displaystyle\int_\Omega |u|^\frac{2\,p}{p+1}\, dx\right)^\frac{1}{p}}-\frac{1}{U}\right\|^{p+1}_{L^{p}(\Omega)}\\
&-\frac{1}{p\,2^{p}}\,\left|\left(\displaystyle\int_\Omega |u|^\frac{2\,p}{p+1}\, dx\right)^\frac{1}{p}-\left(\displaystyle\int_\Omega u^2\, U\, dx\right)^\frac{1}{p+1}\right|^{p+1},
\end{split}
\end{equation}
where as always we used \eqref{stimaLV} and \eqref{dalcanto} to estimate the $L^2(\Omega;U)$ norm of $u$.
By using \eqref{primotempo} and the convexity of $t\mapsto t^{p+1}$, we get
\[
\begin{split}
2\, (\mathcal{E}_f(U)-\mathcal{E}_f(U_0))&\ge \mathcal{I}_1\ge (p+1)\,\int_\Omega |u|^\frac{2\,p}{p+1}\, dx\,\left[\left(\displaystyle\int_\Omega u^2\, U\, dx\right)^\frac{1}{p+1}-\left(\int_\Omega |u|^\frac{2\,p}{p+1}\, dx\right)^\frac{1}{p}\right]\\
&\ge (p+1)\,\left(\frac{c_2^2}{2}\right)^\frac{p}{p+1}\,\left[\left(\displaystyle\int_\Omega u^2\, U\, dx\right)^\frac{1}{p+1}-\left(\int_\Omega |u|^\frac{2\,p}{p+1}\, dx\right)^\frac{1}{p}\right]
\end{split}
\]
By using the latter, from \eqref{boh2} we can infer
\[
\begin{split}
\mathcal{I}_1&\ge \frac{c_2^2}{p\,2^{2\,p+1}}\,\left\|\frac{|u|^\frac{2}{p+1}}{\left(\displaystyle\int_\Omega |u|^\frac{2\,p}{p+1}\, dx\right)^\frac{1}{p}}-\frac{1}{U}\right\|^{p+1}_{L^{p}(\Omega)}\\
&-\frac{2}{p\,(p+1)}\,\left(\frac{2}{(p+1)\, c_2^2}\right)^p\, (\mathcal{E}_f(U)-\mathcal{E}_f(U_0))^{p+1}.
\end{split}
\]
We now insert the previous estimate in \eqref{primotempo}, use that $\mathcal{E}_f(U)-\mathcal{E}_f(U_0)\le 1$ and take the power $1/(p+1)$ on both sides. The resulting estimate is
\begin{equation}
\label{I1fine}
\Big(\mathcal{E}_f(U)-\mathcal{E}_f(U_0)\Big)^\frac{1}{p+1}\ge c_6\, \left\|\frac{|u|^\frac{2}{p+1}}{\left(\displaystyle\int_\Omega |u|^\frac{2\,p}{p+1}\, dx\right)^\frac{1}{p}}-\frac{1}{U}\right\|_{L^{p}(\Omega)},
\end{equation}
where $c_6>0$ is the following constant depending only on $p$ and $c_2$
\begin{equation}
\label{c6}
c_6:=\left(1+\frac{2}{p\,(p+1)}\,\left(\frac{2}{(p+1)\,c_2^{2}}\right)^p\right)^{-\frac{1}{p+1}}\, \left(\frac{c_2^2}{p\, 4^{p+1}}\right)^\frac{1}{p+1}.
\end{equation}
\underline{\it Estimate on $\mathcal{I}_2$.}
Again by combining the triangle inequality and the convexity of $t\mapsto t^2$, we have
\begin{equation}
\label{H11}
\begin{split}
\mathcal{I}_2=\|u-u_0\|^2_{W^{1,2}_0(\Omega)}&=c_2^2\,\left\|\frac{u}{\|u_0\|_{L^\frac{2\,p}{p+1}(\Omega)}}-\frac{u_0}{\|u_0\|_{L^\frac{2\,p}{p+1}(\Omega)}}\right\|^2_{W^{1,2}_0(\Omega)}\\
&\ge\frac{1}{2}\,c_2^2\,\left\|\frac{u}{\|u\|_{L^\frac{2\,p}{p+1}(\Omega)}}-\frac{u_0}{\|u_0\|_{L^\frac{2\,p}{p+1}(\Omega)}}\right\|^2_{W^{1,2}_0(\Omega)}\\
&-\left(\frac{\|u\|_{W^{1,2}_0(\Omega)}}{\|u\|_{L^\frac{2\,p}{p+1}(\Omega)}}\right)^2\,\left(\|u\|_{L^\frac{2\,p}{p+1}(\Omega)}-\|u_0\|_{L^\frac{2\,p}{p+1}(\Omega)}\right)^2.
\end{split}
\end{equation}
We also observe that, by recalling the energy estimate \eqref{energyestimate} and \eqref{dalcanto}, we get
\begin{equation}
\label{mare}
\left(\frac{\|u\|_{W^{1,2}_0(\Omega)}}{\|u\|_{L^\frac{2\,p}{p+1}(\Omega)}}\right)^2\le \frac{2}{c_2^2}\,\|f\|^2_{W^{-1,2}(\Omega)}.
\end{equation}
In order to estimate the negative term on the right-hand side in \eqref{H11}, we can simply use Lemma \ref{lm:corner}. Indeed, we have
\begin{equation}
\label{normaLmerda_a}
\begin{split}
\left|\|u\|_{L^\frac{2\,p}{p+1}(\Omega)}-\|u_0\|_{L^\frac{2\,p}{p+1}(\Omega)}\right|&\le \left|\|u\|_{L^\frac{2\,p}{p+1}(\Omega)}-\|u_0\|_{L^\frac{2\,p}{p+1}(\Omega)}\right|\frac{ \|u\|_{L^\frac{2\,p}{p+1}(\Omega)}+\|u_0\|_{L^\frac{2\,p}{p+1}(\Omega)} }{\|u_0\|_{L^\frac{2\,p}{p+1}(\Omega)}}\\
&=\frac{1}{\|u_0\|_{L^\frac{2\,p}{p+1}(\Omega)}}\,\left|\|u\|^2_{L^\frac{2\,p}{p+1}(\Omega)}-\|u_0\|^2_{L^\frac{2\,p}{p+1}(\Omega)}\right|\\
&\le \frac{1}{c_3\, c_2}\sqrt{\mathcal{E}_f(U)-\mathcal{E}_f(U_0)}
\end{split}
\end{equation}
where we used \eqref{normaLmerda} (recall that we are assuming \eqref{piccolo}). By using \eqref{mare} and \eqref{normaLmerda_a} in \eqref{H11}, a further application of Sobolev inequality leads us to
\begin{equation}
\label{I2fine}
\sqrt{\mathcal{E}_f(U)-\mathcal{E}_f(U_0)}\ge c_7\,\left\|\frac{u}{\|u\|_{L^\frac{2\,p}{p+1}(\Omega)}}-\frac{u_0}{\|u_0\|_{L^\frac{2\,p}{p+1}(\Omega)}}\right\|_{L^{2^*}(\Omega)},
\end{equation}
where the constant $c_7>0$ depends only on $N,\,f$ and $c_2$
\begin{equation}
\label{c7}
c_7:=\sqrt{\frac{T_N}{c_3^2\,c_2^4+\|f\|^2_{W^{-1,2}(\Omega)}}}\,\frac{c_3\,c^3_2}{2}.
\end{equation}
\underline{\it Stability estimate for $U$.} We now use Lemma \ref{lm:sospetto} with the choices
\[
r=\frac{2\,p}{p+1},\qquad\qquad s=2^*,
\]
and
\begin{equation}
\label{UU0}
g_0=\frac{u_0}{\|u_0\|_{L^\frac{2\,p}{p+1}(\Omega)}},\qquad\qquad g=\frac{1}{2}\left[\frac{u}{\|u\|_{L^\frac{2\,p}{p+1}(\Omega)}}+\frac{u_0}{\|u_0\|_{L^\frac{2\,p}{p+1}(\Omega)}}\right].
\end{equation}
Thus we get
\[
\left\|g\right\|_{L^\frac{2\,p}{p+1}(\Omega)}\ge 1-c_4\,c_2^\frac{1-p}{p+1}\,\left\|g-g_0\right\|_{L^{2^*}(\Omega)},
\]
since
\[
\frac{\left\||g_0|^{r-1}\right\|_{L^{s'}(\Omega)}}{\|g_0\|^{r-1}_{L^r(\Omega)}}=\left\||g_0|^\frac{p-1}{p+1}\right\|_{L^{(2^*)'}(\Omega)}=c_4\, c_2^\frac{1-p}{p+1},
\]
which is finite by hypothesis.
By combining the previous with \eqref{I2fine}, we obtain\footnote{Up to further suppose that
\[
\mathcal{E}_f(U)-\mathcal{E}_f(U_0)\le \min\left\{c_5,\left(\frac{2\,c_7\, c_2^\frac{p-1}{p+1}}{c_4}\right)^2\right\},
\]
we can assume that the right-hand side of \eqref{uff!} is positive.}
\begin{equation}
\label{uff!}
\left\|g\right\|_{L^\frac{2\,p}{p+1}(\Omega)}\ge 1-\frac{c_4\,c_2^\frac{1-p}{p+1}}{2\,c_7}\sqrt{\mathcal{E}_f(U)-\mathcal{E}_f(U_0)}.
\end{equation}
The previous estimate is crucial in order to estimate $g-g_0$ in $L^{2\,p/(p+1)}$. Indeed, by {\it Clarkson inequality}\footnote{Let $1<q\le 2$ and $h_1,h_2\in L^q(\Omega)$ be two functions with unit norm. Then we have
\[
\left\|\frac{h_1+h_2}{2}\right\|^{q'}_{L^q(\Omega)}+\left\|\frac{h_1-h_2}{2}\right\|^{q'}_{L^q(\Omega)}\le 1,
\]
see \cite[Theorem 2]{Cl}.} and \eqref{uff!} we can infer
\[
\left\|g-g_0\right\|^{\frac{2\,p}{p-1}}_{L^\frac{2\,p}{p+1}(\Omega)}\le 1-\left\|g\right\|^\frac{2\,p}{p-1}_{L^\frac{2\,p}{p+1}(\Omega)}\le 1-\left(1-\frac{c_4\,c_2^\frac{1-p}{p+1}}{2\,c_7}\,\sqrt{\mathcal{E}_f(U)-\mathcal{E}_f(U_0)}\right)^\frac{2\,p}{p-1},
\]
and thus
\[
\left\|g-g_0\right\|^{\frac{2\,p}{p-1}}_{L^\frac{2\,p}{p+1}(\Omega)}\le \frac{p}{p-1}\,\frac{c_4\,c_2^\frac{1-p}{p+1}}{c_7}\,\sqrt{\mathcal{E}_f(U)-\mathcal{E}_f(U_0)},
\]
thanks to the convexity of $t\mapsto t^{2\,p/(p-1)}$. 
We now go back to the definition \eqref{UU0} of $g$ and $g_0$, so that the previous finally gives
\begin{equation}
\label{inLcosa}
\left(\frac{1}{2}\right)^\frac{2\,p}{p-1}\,\left\|\frac{u}{\|u\|_{L^\frac{2\,p}{p+1}(\Omega)}}-\frac{u_0}{\|u_0\|_{L^\frac{2\,p}{p+1}(\Omega)}}\right\|^{\frac{2\,p}{p-1}}_{L^\frac{2\,p}{p+1}(\Omega)}\le \frac{p}{p-1}\,\frac{c_4\,c_2^\frac{1-p}{p+1}}{c_7}\,\sqrt{\mathcal{E}_f(U)-\mathcal{E}_f(U_0)},
\end{equation}
Recall that $2/(p+1)<1$, thus the function $t\mapsto |t|^{2/(p+1)}$ is $2/(p+1)-$H\"older continuous and we have
\[
\left|\frac{|u|^\frac{2}{p+1}}{\left(\displaystyle\int_\Omega |u|^\frac{2\,p}{p+1}\,dx\right)^\frac{1}{p}}-\frac{|u_0|^\frac{2}{p+1}}{\left(\displaystyle\int_\Omega |u_0|^\frac{2\,p}{p+1}\,dx\right)^\frac{1}{p}}\right|^p\le \left|\frac{u}{\|u\|_{L^\frac{2\,p}{p+1}(\Omega)}}-\frac{u_0}{\|u\|_{L^\frac{2\,p}{p+1}(\Omega)}}\right|^\frac{2\,p}{p+1}.
\]
Thus from \eqref{inLcosa} we obtain
\begin{equation}
\label{penultimo}
\big[\mathcal{E}_f(U)-\mathcal{E}_f(U_0)\big]^\frac{p-1}{2\,p\,(p+1)}\ge c_8\,\left\|\frac{|u|^\frac{2}{p+1}}{\left(\displaystyle\int_\Omega |u|^\frac{2\,p}{p+1}\,dx\right)^\frac{1}{p}}-\frac{|u_0|^\frac{2}{p+1}}{\left(\displaystyle\int_\Omega |u_0|^\frac{2\,p}{p+1}\,dx\right)^\frac{1}{p}}\right\|_{L^p(\Omega)},
\end{equation}
where the constant $c_8$ is given by
\begin{equation}
\label{c8}
c_8:=\left(\frac{1}{4}\,\left(\frac{p-1}{p}\,\frac{c_7\,c_2^\frac{p-1}{p+1}}{c_4}\right)^\frac{p-1}{p}\right)^\frac{1}{p+1}.
\end{equation}
If we now use the relation \eqref{relazionem} between $U_0$ and $u_0$, the triangle inequality, \eqref{I1fine} and \eqref{penultimo} we get
\[
\begin{split}
\left\|\frac{1}{U}-\frac{1}{U_0}\right\|_{L^p(\Omega)}&\le \left\|\frac{1}{U}-\frac{|u|^\frac{2}{p+1}}{\left(\displaystyle\int_\Omega |u|^\frac{2\,p}{p+1}\,dx\right)^\frac{1}{p}}\right\|_{L^p(\Omega)}+\left\|\frac{|u|^\frac{2}{p+1}}{\left(\displaystyle\int_\Omega |u|^\frac{2\,p}{p+1}\,dx\right)^\frac{1}{p}}-\frac{|u_0|^\frac{2}{p+1}}{\left(\displaystyle\int_\Omega |u_0|^\frac{2\,p}{p+1}\,dx\right)^\frac{1}{p}}\right\|_{L^p(\Omega)}\\
&\le \left[\frac{\Big(\mathcal{E}_f(U)-\mathcal{E}_f(U_0)\Big)^\frac{1}{p+1}}{c_6}+\frac{\big(\mathcal{E}_f(U)-\mathcal{E}_f(U_0)\big)^\frac{p-1}{2\,p\,(p+1)}}{c_8}\right].
\end{split}
\]
Observe that since $p>1$, we have
\[
\frac{p-1}{2\,p\,(p+1)}<\frac{1}{p+1},
\]
thus it is now sufficient to use hypothesis \eqref{piccolo} to get \eqref{eccolo} for $U$.
\vskip.2cm\noindent
\underline{\it Conclusion: stability estimate for $V$.} We now go back to our potential $V$. By using \eqref{reduction2}, the previous step and \eqref{reduction1}, we get
\[
\begin{split}
\left\|\frac{1}{V}-\frac{1}{U_0}\right\|_{L^p(\Omega)}&\le \left\|\frac{1}{U}-\frac{1}{U_0}\right\|_{L^p(\Omega)}+\frac{4}{c_2^2}\, \Big(\mathcal{E}_f(V)-\mathcal{E}_f(U_0)\Big)\\
&\le \frac{c_6+c_8}{c_6\,c_8}\, \Big(\mathcal{E}_f(U)-\mathcal{E}_f(U_0)\Big)^\frac{p-1}{2\,p\,(p+1)}+\frac{4}{c_2^2}\, \Big(\mathcal{E}_f(V)-\mathcal{E}_f(U_0)\Big)\\
&\le \left[\frac{c_6+c_8}{c_6\,c_8}+\frac{4}{c_2^2}\right]\,\Big(\mathcal{E}_f(V)-\mathcal{E}_f(U_0)\Big)^\frac{p-1}{2\,p\,(p+1)}, 
\end{split}
\]
where we also used that $\mathcal{E}_f(V)-\mathcal{E}_f(U_0)\le 1$. This concludes the proof.
\end{proof}
Some comments on the previous result are in order.
\begin{oss}[Integrability assumption on $u_0$]
\label{oss:ipotesi}
We point out that
\[
0<(2^*)'\, \frac{p-1}{p+1}<2,\qquad \mbox{ for every } p>1,
\]
thus the condition \eqref{ipotesiu0}
of Theorem \ref{teo:min_stab} is always satisfied if $|\Omega|<+\infty$. When $|\Omega|=+\infty$, this is still verified if $f$ decreases sufficiently fast at infinity. For example, by appealing to Theorem \ref{teo:u0} in the Appendix this holds true for $f\in L^{r}(\Omega)$ with $r>N/2$ and 
\[
|f(x)|=O\left(\frac{1}{|x|^\alpha}\right),\qquad \mbox{ for } |x|\to \infty,\qquad \alpha>\frac{N+2}{2}.
\]
Observe that the condition on $\alpha$ is the minimal assumption for $|x|^{-\alpha}$ to be $(2^*)'-$integrable at infinity.
\end{oss}
\begin{oss}
\label{oss:sigma}
A closer inspection of the previous proof informs us that a possible value for the constant $\sigma_m$ in \eqref{eccolo} is 
\[
\sigma_m=\min\left\{\left(\frac{c_6+c_8}{c_6\,c_8}+\frac{4}{c_2^2}\right)^{-\frac{2\,p\,(p+1)}{p-1}},\ c_5\,4^{-\frac{p\,(p+1)}{p-1}},\ \left(\frac{2\,c_7\, c_2^\frac{p-1}{p+1}}{c_4}\right)^2 4^{-\frac{p\,(p+1)}{p-1}}\right\}.
\]
where the constants $c_5,\,c_6,\,c_7$ and $c_8$ are defined in \eqref{c5}, \eqref{c6}, \eqref{c7} and \eqref{c8}.
\end{oss}

\subsection{Stability of the state functions}
By suitably combining some of the estimates we used so far, we also get a stability result for the energy functions in the natural space $W^{1,2}_0(\Omega)\cap L^{2\,p/(p+1)}(\Omega)$.
\begin{prop}
Under the hypotheses of Theorem \ref{teo:min_stab}, we have
\begin{equation}
\label{stab_functions}
\Big(\mathcal{E}_f(V)-\mathcal{E}_f(V_0)\Big)^\frac{p-1}{2\,p}\ge c\left[ \|u-u_0\|^2_{W^{1,2}_0(\Omega)}+\left\|u-u_0\right\|^{2}_{L^\frac{2\,p}{p+1}(\Omega)}\right],
\end{equation}
for some constant $c>0$ depending on $N,p,c_2,c_4$ and $\|f\|_{W^{-1,2}(\Omega)}$.
\end{prop}
\begin{proof}
We first observe that by \eqref{energyestimate} we have
\[
\|u-u_0\|_{W^{1,2}_0(\Omega)}\le \|u\|_{W^{1,2}_0(\Omega)}+\|u_0\|_{W^{1,2}_0(\Omega)}\le2\,\|f\|_{W^{-1,2}(\Omega)},
\]
and by \eqref{energyestimate2}
\[
\|u-u_0\|_{L^\frac{2\,p}{p+1}(\Omega)}\le \|u\|_{L^\frac{2\,p}{p+1}(\Omega)}+\|u_0\|_{L^\frac{2\,p}{p+1}(\Omega)}\le\sqrt{2}\,\|f\|_{W^{-1,2}(\Omega)}.
\]
Thus we can assume without loss of generality that
\[
\mathcal{E}_f(V)-\mathcal{E}_f(V_0)\le 1,
\]
otherwise the result is trivially true. From \eqref{fundamental} and \eqref{cazzatina} we already know 
\[
\mathcal{E}_f(V)-\mathcal{E}_f(V_0)\ge\frac{1}{2}\,\|u-u_0\|^2_{W^{1,2}_0(\Omega)}
\]
and by \eqref{inLcosa}
\[
c_9\,\left\|\frac{u}{\|u\|_{L^\frac{2\,p}{p+1}(\Omega)}}-\frac{u_0}{\|u_0\|_{L^\frac{2\,p}{p+1}(\Omega)}}\right\|_{L^\frac{2\,p}{p+1}(\Omega)}^{\frac{4\,p}{p-1}}\le\mathcal{E}_f(V)-\mathcal{E}_f(V_0),
\]
with
\[
c_9=\left(\frac{p-1}{p}\,\frac{c_7\,c_2^\frac{p-1}{p+1}}{c_4}\right)^2\,\left(\frac{1}{2}\right)^\frac{4\,p}{p-1}.
\]
Moreover, by the triangle inequality it is not difficult to see that
\[
\begin{split}
\left\|u-u_0\right\|_{L^\frac{2\,p}{p+1}(\Omega)}&\le \left|\|u\|_{L^\frac{2\,p}{p+1}(\Omega)}-\|u_0\|_{L^\frac{2\,p}{p+1}(\Omega)}\right|\\
&+c_2\,\left\|\frac{u}{\|u\|_{L^\frac{2\,p}{p+1}(\Omega)}}-\frac{u_0}{\|u_0\|_{L^\frac{2\,p}{p+1}(\Omega)}}\right\|_{L^\frac{2\,p}{p+1}(\Omega)}.
\end{split}
\]
By combining these estimates and using \eqref{normaLmerda_a}, we get the desired conclusion.
\end{proof}
\begin{oss}
By interpolation, it is easy to obtain a stability estimate like \eqref{stab_functions} in $L^r(\Omega)$ for every $2\,p/(p+1)<r<2^*$ and in $W^{s,2}_0(\Omega)$ for every $0<s<1$.
\end{oss}

\appendix

%%%%%%%%%%%%%%%%%%%%%%%%%%%%%%%%%%%%%%%%%%%%%%%%%%
\section{Sharp decay estimates for non autonomous Schr\"odinger equations}\label{sec:decay}

Given $1<p<\infty$, we set for simplicity $q=2\,p/(p+1)$ which is always between $1$ and $2$. In what follows we still denote by $u_0$ the unique minimizer of
\[
J_{p,f}(u)=\frac{1}{2}\, \int_\Omega |\nabla u|^2\, dx+\frac{1}{2}\, \left(\int_\Omega |u|^q\, dx\right)^\frac{2}{q}\,-\langle f,u\rangle.
\]
The aim of this Appendix is to prove some decay properties for the optimal function $u_0$, in the case $|\Omega|=+\infty$. We can confine ourselves to consider the case $\Omega=\mathbb{R}^N$.
\begin{teo}[Properties of $u_0$]
\label{teo:u0}
Let $r>N/2$ and $f\in L^{r}(\mathbb{R}^N)$ be such that there exist $C,R>0$ and $\alpha>(N+2)/2$ with
\[
|f(x)|\le C\, |x|^{-\alpha},\qquad \mbox{ for } |x|\ge R. 
\]
Then there exists $M=M(\|f\|_{L^r(\mathbb{R}^N)},c_2,C,R,\alpha)>0$ such that 
\begin{equation}
\label{linfty}
|u_0(x)|\le M,\qquad x\in\mathbb{R}^N.
\end{equation}
Moreover, if we denote by $w\in W^{1,2}_0(\mathbb{R}^N)$ the unique minimizer of
\[
\mathcal{J}(u)=\frac{1}{2}\, \int_{\mathbb{R}^N} |\nabla u|^2\, dx+\frac{c_2^{2-q}}{q}\, \int_{\mathbb{R}^N} |u|^{q}\, dx-\int_{\mathbb{R}^N} \frac{u}{(1+|x|^2)^{\alpha/2}}\, dx,
\]
then there exists $T=T(M,C,R,\alpha)>0$ such that
\begin{equation}
\label{comparison}
|u_0(x)|\le T^{2/(2-q)}\, w\left(\frac{x}{T}\right),\qquad x\in\mathbb{R}^N.
\end{equation}
In particular, we get
\begin{equation}
\label{bum}
|u_0(x)|\le C'\, |x|^{-\alpha/(q-1)},\qquad \mbox{ for } |x|\gg 1,
\end{equation}
for some constant $C'>0$.
\end{teo}
\begin{proof}
We first observe that if $f \ge0$, then the unique minimizer $u_0$ is positive since $J_{p,f}(|u|)\le J_{p,f}(u)$. We also notice that it is not restrictive  to prove the result for $f\ge0$. Indeed, if $f$ is not positive, by using the minimality of $u_0$ and the fact that $f\le |f|$, it is not difficult to see that
\[
|u_0|\le \widetilde u_0,
\]
where $\widetilde u_0$ is the unique minimizer of $J_{p,|f|}$.
We thus assume $f\ge 0$ in what follows and divide the proof in three parts.
\vskip.2cm\noindent
{\it Boundedness of $u_0$}. 
The integrability of $f$ already implies that $u_0\in L^\infty_{loc}(\mathbb{R}^N)$, see \cite[Chapter 7]{Gi}. Also, since $u_0$ solves
\[
-\Delta u_0+c_2^{2-q}\, u_0^{q-1}=f,
\]
it is the unique minimizer of the functional
\[
\widetilde J_{p,f}(u)=\frac{1}{2}\, \int_{\mathbb{R}^N} |\nabla u|^2\, dx+\frac{c_2^{2-q}}{q}\, \int_{\mathbb{R}^N} |u|^{q}\, dx-\int_{\mathbb{R}^N} f\,u\, dx,
\]
as well. Let $M>0$, by testing the minimality of $u_0$ against $\varphi_M=\min\{u_0,M\}$, we get 
\[
\begin{split}
\widetilde J_{p,f}(\varphi_M)-\widetilde J_{p,f}(u_0)&\le \frac{c_2^{2-q}}{q}\,\int_{\{u_0>M\}} M^{q}\, dx-\frac{c_2^{2-q}}{q}\, \int_{\{u_0> M\}} u_0^{q}\, dx+\int_{\{u_0>M\}} f\, (u_0-M)\, dx\\
&= -\frac{c_2^{2-q}}{q}\, \int_{\{u_0>M\}} \left(u_0^q-M^q\right)\, dx+\int_{\mathbb{R}^N} f\, (u_0-M)_+\, dx.\\
\end{split}
\]
We then observe that 
\[
u_0^{q}-M^q\ge q\,M^{q-1}\, (u_0-M).
\]
If we take 
\[
M=\max\left\{\|u_0\|_{L^\infty(B_R)},\, \left(c_2^{q-2}\,C\, R^{-\alpha}\right)^{1/(q-1)}\right\},
\]
then we have
\[
\{u_0>M\}\subset \{x\, :\, |x|\ge R\},
\]
so that
\[
f(x)\le C\, |x|^{-\alpha}\le C\, R^{-\alpha}\le c_2^{2-q}\, M^{q-1},\qquad \mbox{ on } \{u_0>M\}.
\] 
In conclusion, by using the choice of $M$ and the decay of $f$, we get
\[
\widetilde J_{p,f}(\varphi_M)-\widetilde J_{p,f}(u_0)\le \int_{\mathbb{R}^N} \left(f-c_2^{2-q}\, M^{q-1}\right)\, (u_0-M)_+\, dx\le 0. 
\]
By uniqueness of the minimizer of $u_0$ we get that $\varphi_M=u_0$ and thus $u_0\le M$ in $\mathbb{R}^N$.
\vskip.2cm\noindent
{\it Comparison}. In order to prove the second assertion, we start observing that $\alpha>(N+2)/2$ guarantees
\[
h(x)=\left(1+|x|^2\right)^{-\alpha/2}\in L^{(2^*)'}(\mathbb{R}^N)\subset W^{-1,2}(\mathbb{R}^N),
\]
thus a function $w$ minimizing $\mathcal{J}$ exists, is unique and radially decreasing. Moreover, it solves
\[
-\Delta w+c_2^{2-q}\, w^{q-1}=h,\qquad \mbox{ in }\mathbb{R}^N.
\]
The rescaled function 
\[
w_t(x)=t^{2/(2-q)}\, w\left(\frac{x}{t}\right),\qquad t>0,
\]
then solves
\[
-\Delta w_t+c_2^{2-q}\,w_t^{q-1}=h_t,\qquad \mbox{ where }\quad h_t(x)=t^{2\,(q-1)/(2-q)}\, h\left(\frac{x}{t}\right).
\]
Since by the first part of the proof $u_0$ is bounded and $w_{t_1}\ge w_{t_0}$ for $t_1\ge t_0$, we can find a $T_0$ sufficiently large such that 
\begin{equation}
\label{wtu0}
w_{t}(x)\ge u_0(x),\qquad \mbox{ for } |x|\le R \quad \mbox{ and }\quad\ t\ge T_0.
\end{equation}
In addition, if we define
\[
T_1=\max\left\{R,(2^{\alpha/2}\, C\, R^{-\alpha})^{(2-q)/(2\,q-2)}\right\},
\]
by hypothesis on $f$ we get\footnote{For the first inequality, we just use that for $x,t\ge R$ we have
\[
\left(1+\frac{|x|^2}{t^2}\right)^\frac{\alpha}{2}\le \left(1+\frac{|x|^2}{R^2}\right)^\frac{\alpha}{2}\le 2^{\alpha/2}\, \frac{|x|^\alpha}{R^\alpha}.
\]}
\begin{equation}
\label{htf}
h_t(x)=\frac{t^{\frac{2\,q-2}{2-q}+\alpha}}{\left(t^2+|x|^2\right)^{\alpha/2}}\ge \frac{t^\frac{2\,q-2}{2-q}\,R^\alpha}{2^{\alpha/2}\,|x|^\alpha}\ge \frac{C}{|x|^\alpha}\ge f(x),\qquad \mbox{ for } |x|>R\quad \mbox{ and }\quad t\ge T_1.
\end{equation}
We now define $T=\max\{T_0,T_1\}$ and test the minimality of $u_0$ against the function $\psi_{T}=\min\{u_0,w_{T}\}$. Thus we get
\[
\begin{split}
0\ge \widetilde J_{p,f}(u_0)-\widetilde J_{p,f}(\psi_{T})&\ge \frac{1}{2}\, \int_{\{w_{T}<u_0\}} \left(|\nabla u_0|^2-|\nabla w_{T}|^2\right)\, dx\\
&+\frac{c_2^{2-q}}{q}\, \int_{\{w_{T}<u_0\}} \left(u_0^q-w_{T}^q\right)\, dx\\
&-\int_{\{w_{T}<u_0\}} f\, (u_0-w_{T})\, dx.
\end{split}
\]
By using the convexity of the functions involved and the equation solved by $w_{T}$, we thus get
\[
\begin{split}
0\ge \int_{\mathbb{R}^N} \langle \nabla w_{T},\nabla (u_0-w_{T})_+\rangle\,& dx+c_2^{2-q}\, \int_{\mathbb{R}^N} w_T^{q-1}\, (u_0-w_{T})_+\, dx\\
&-\int_{\mathbb{R}^N} f\, (u_0-w_{T})_+\, dx=\int_{\mathbb{R}^N} (h_T-f)\, (u_0-w_{T})_+\, dx.
\end{split} 
\]
By combining \eqref{wtu0} and \eqref{htf}, we have
\[
h_T(x)\ge f(x),\qquad \mbox{ on } \{u_0>w_T\},
\]
thus we would obtain that $\psi_T$ is a minimizer of $\widetilde J_{p,f}$. By uniqueness, $\psi_T=u_0$ and thus \eqref{comparison} holds true.
\vskip.2cm\noindent
{\it Decay estimate for $u_0$}. Finally, the estimate \eqref{bum} simply follows from \eqref{comparison} and Lemma \ref{lm:decay} below, applied to the rescaled function $w_T$.
\end{proof}
\begin{oss}
A different way to compare $u_0$ with a radial function and obtain \eqref{bum} could be that of using symmetrization techniques. More precisely, one could look at the radial solution of the {\it symmetrized problem}
\[
-\Delta v+c_2^{2-q}\, v^{q-1}=f^*,
\] 
where $f^*$ denotes the {\it Schwarz rearrangement} of $f$ (see \cite[Chapter 2]{hen06} for the relevant definition). There is a huge literature on results which permit to compare $u_0^*$ and $v$ (see for example \cite{ALT,Ta}), but the presence of the nonlinear term $c_2^{2-q}\, v^{q-1}$ complicates the task. An interesting result covering this case is contained in the recent paper \cite{HR} by Hamel and Russ, which however deals with the case of a {\it bounded domain} $\Omega$. Since it is not clear whether this strategy could work or not, we decided to take a different path, which just uses the minimality of $u_0$. We also refer to the related discussion in \cite[Sections 5.2 and 5.3]{Va}.
\end{oss}
\begin{lm}[Sharp decay estimate]
\label{lm:decay}
Let $N\ge 3$, $1<q\le 2$ and 
\(
\alpha>(N+2)/2.
\)
Let us suppose that $u\in W^{1,2}_0(\mathbb{R}^N)\cap L^{q}(\mathbb{R}^N)$ is a smooth positive radial function verifying
\[
-\Delta u+a\, u^{q-1}\le b\,|x|^{-\alpha},\qquad \mbox{ for } |x|\ge R,
\]
for some $a,R>0$ and $b\ge 0$.
Then there exists a constant $C>0$ depending on $\alpha,q,R,a,\,\|u\|_{W^{1,2}_0(\mathbb{R}^N)}$ and $\|u\|_{L^q(\mathbb{R}^N)}$ such that 
\begin{equation}
\label{olè}
0\le u(x)\le C\, |x|^{-\alpha/(q-1)},\qquad \mbox{ for } |x|\gg 1.
\end{equation}
\end{lm}
\begin{proof}
We divide the proof in two parts: in the first we prove that if 
\[
-\Delta u +a\, u^{q-1}\le b\, |x|^{-\gamma},\qquad |x|\ge R,
\]
for $\gamma>(N+2)/2$, then the following weaker decay estimate holds
\begin{equation}
\label{weakdecay}
0\le u(x)\le C_\varepsilon\, |x|^{-\gamma+\varepsilon},\qquad \mbox{ for } |x|\gg 1,
\end{equation}
for every $\varepsilon>0$.
Then we will get \eqref{olè} by using a contradiction argument and a suitable maximum principle.
\vskip.2cm\noindent
\underline{\it Part 1: weak decay.}
We first observe that Lemma \ref{lm:strauss} below already implies that
\[
u(\varrho)\lesssim \varrho^{-\beta_0},\qquad \mbox{ where }\quad\beta_0=(N-1)\,\frac{2}{2+q}.
\]
If $\gamma\le \beta_0$, then \eqref{weakdecay} holds and there is nothing to prove. We can thus assume that $\gamma>\beta_0$.
We are going to prove \eqref{weakdecay} by using a recursive argument, namely we will prove the following implication:
\begin{equation}
\label{boothstrap}
|u(\varrho)|\le c\,\varrho^{-\beta},\quad \mbox{ for }\varrho>r_0 \mbox{ and } \beta\ge \beta_0\qquad\Longrightarrow\qquad |u(\varrho)|\le c'\, \varrho^{-\frac{\gamma+\beta}{2}},\quad \mbox{ for } \varrho>r_0'.
\end{equation}
By starting from $\beta=\beta_0$ and iterating a finite number of times \eqref{boothstrap}, 
we will get the desired result. Indeed, observe that since $\gamma>\beta_0$, the sequence 
\[
\beta_{i+1}=\frac{\gamma+\beta_i}{2}
\]
is monotone increasing and converges to $\gamma$. 
\par
To prove \eqref{boothstrap} we adapt a classical argument that can be found for example in \cite[Lemma 2]{BL}, but some care is needed in order to deal with the non-autonomous term. Also, for notational simplicity we give the proof just for $a=b=1$.
Then by hypothesis we have that $u$ verifies
\[
u''(\varrho)+\frac{N-1}{\varrho}\, u'(\varrho)-u(\varrho)^{q-1}\ge -\varrho^{-\gamma},\qquad \varrho\ge R.
\]
The function $v(\varrho)=\varrho^{(N-1)/2}\, u(\varrho)$ verifies
\[
v''\ge \left[\frac{C_N}{\varrho^2}+(u+1)^{q-2}\right]\, v-\varrho^{\frac{N-1}{2}-\gamma},\qquad \varrho\ge R.
\]
We then make the further substitution $w=v^2$
\[
\frac{w''}{2}\ge (v')^2+w\,\left[\frac{C_N}{\varrho^2}+(1+u)^{q-2}\right]-\varrho^{\frac{N-1}{2}-\gamma}\, \sqrt{w},\qquad \varrho\ge R.
\]
As we already know that $u\to 0$ as $\varrho \to 0$, we have
\[
\frac{C_N}{\varrho^2}+(1+u)^{q-2}\ge \frac{m}{2},\qquad \varrho\ge  R,
\]
where $m>0$ is a suitable constant. 
Thus we obtain that $w$ verifies
\begin{equation}
\label{ODE}
w''-m\,w\ge -2\,\varrho^{\frac{N-1}{2}-\gamma}\, \sqrt{w},\qquad \varrho\ge R.
\end{equation}
We then set
\[
z(\varrho)=e^{-\sqrt{m}\,\varrho}\, \left(w'(\varrho)+\sqrt{m}\,w(\varrho)\right),\qquad \varrho>R,
\]
thus we get
\[
z'(\varrho)=e^{-\sqrt{m}\, \varrho}\, \left(w''(\varrho)-m\,w(\varrho)\right)\ge -2\,e^{-\sqrt{m}\,\varrho}\,\varrho^{\frac{N-1}{2}-\gamma}\, \sqrt{w},\qquad \varrho>R.
\]
where we used \eqref{ODE}. In order to prove \eqref{boothstrap}, we assume that 
\[
|u(\varrho)|\le c\, \varrho^{-\beta},\qquad \varrho>r_0\ge R,
\] 
for $\beta\ge \beta_0$, then by recalling that $w=\varrho^{N-1}\, u^2$ we get
\[
\sqrt{w(\varrho)}\le c\,\varrho^{\frac{N-1}{2}-\beta},\qquad \varrho>r_0.
\]
Thus for $z'$ we can infer
\begin{equation}
\label{zbelow}
z'(\varrho)\ge -c\,e^{-\sqrt{m}\,\varrho}\, \varrho^{N-1-\gamma-\beta},\qquad \varrho>r_0,
\end{equation}
Let us now take the (negative) function $\eta$ defined by
\[
\eta(\varrho)=-c\,\int_\varrho^{+\infty} \left[e^{-\sqrt{m}\,s}\, s^{N-1-\gamma-\beta}\right]\, ds,\qquad \varrho>r_0.
\]
Observe that $\eta$ is strictly increasing and $\eta$ goes to $0$ as $\varrho$ goes to $\infty$.
Then from \eqref{zbelow} we get
\[
z'(\varrho)\ge -c\,e^{-\sqrt{m}\,\varrho}\, \varrho^{N-1-\gamma-\beta}= -\eta'(\varrho),\qquad \varrho>r_0,
\]
that is $z+\eta$ is non decreasing on $(r_0,+\infty)$. Let us suppose that there exists $r_1>r_0$ such that
\[
z(r_1)+\eta(r_1)>0,
\]
then by monotoncitiy of $z+\eta$ we obtain
\[
e^{\sqrt{m}\,\varrho}\big(z(\varrho)+\eta(\varrho)\big)\ge e^{\sqrt{m}\,\varrho}\,\big( z(r_1)+\eta(r_1)\big)>0,\qquad \varrho>r_1.
\]
The previous gives a contradiction, since the right-hand side is not integrable on $(r_1,+\infty)$, while the left-hand side is. Indeed, observe that\footnote{We have that $w=\varrho^{N-1}\, u^2$ which is integrable near $\infty$ since $u\in W^{1,2}_0(\mathbb{R}^N)\cap L^q(\mathbb{R}^N)\subset L^2(\mathbb{R}^N)$. On the other hand, $w'(\varrho)=2\,v(\varrho)\, v'(\varrho)$, with
\[
v(\varrho)=\sqrt{w(\varrho)}\in L^2((r_1,+\infty)),
\]
and
\[
|v'(\varrho)|\lesssim \varrho^\frac{N-3}{2}\, |u(\varrho)|+\varrho^\frac{N-1}{2}\, |u'(\varrho)|\in L^2((r_1,+\infty)).
\]}
\[
e^{\sqrt{m}\, \varrho} \, z(\varrho)=w'(\varrho)+\sqrt{m}\,w(\varrho)\in L^1((r_1,+\infty)),
\]
and $e^{\sqrt{m}\,\varrho}\, \eta$ is integrable at infinity by construction, since
\begin{equation}
\label{uh!}
e^{\sqrt{m}\,\varrho}\, \eta(\varrho)\simeq \frac{\displaystyle\int_\varrho^{+\infty} \left[e^{-\sqrt{m}\,s}\, s^{N-1-\gamma-\beta}\right]\, ds}{e^{-\sqrt{m}\,\varrho}}\simeq  \varrho^{N-1-\gamma-\beta},
\end{equation}
and the latter is integrable on $(r_1,+\infty)$ thanks to the fact that $\gamma> (N+2)/2$ and $\beta\ge \beta_0$. From the previous argument, we get
\[
z(\varrho)\le -\eta(\varrho),\qquad \varrho>r_0,
\]
that is
\[
\left(e^{\sqrt{m}\,\varrho}\, w(\varrho)\right)'=e^{\sqrt{m}\, \varrho} \, \big(w'(\varrho)+\sqrt{m}\, w(\varrho)\big)=e^{2\,\sqrt{m}\,\varrho}\, z(\varrho)\le -e^{2\,\sqrt{m}\,\varrho}\,\eta(\varrho),\qquad \varrho>r_0.
\]
This in turn implies
\[
0\le w(\varrho)\le e^{-\sqrt{m}\, \varrho}\left[C-\int_{R}^{\varrho} e^{2\,\sqrt{m}\,s}\,\eta(s)\, ds\right],\qquad \varrho>r_0.
\]
By recalling the definition of $\eta$, we get
\[
e^{-\sqrt{m}\,\varrho}\, \int_{R}^{\varrho} e^{2\,\sqrt{m}\,s}\,\eta(s)\, ds\simeq \frac{e^{2\,\sqrt{m}\,\varrho}\, \eta(\varrho)}{e^{\sqrt{m}\,\varrho}}\simeq e^{\sqrt{m}\,\varrho}\,\eta(\varrho)\simeq \varrho^{N-1-\gamma-\beta},\qquad \varrho\gg 1,
\] 
thanks to \eqref{uh!}. This finally implies the following decay of $w$ at infinity
\[
w(\varrho)\le C\, \varrho^{N-1-\gamma-\beta},\qquad \varrho\gg 1,
\]
and by going back to $u$, we can finally infer
\[
u(\varrho)=\sqrt{w}\, \varrho^\frac{1-N}{2}\le \sqrt{C}\, \varrho^\frac{1-N}{2}\,\varrho^{\frac{N-1}{2}-\frac{\gamma}{2}-\frac{\beta}{2}}\,=\sqrt{C}\, \varrho^{-\frac{\gamma+\beta}{2}},\qquad \varrho\gg 1.
\]
This concludes the proof of \eqref{boothstrap} and thus of \eqref{weakdecay}, as already explained.
\vskip.2cm\noindent
\underline{\it Part 2: sharp decay.} We now prove \eqref{olè}. Let us assume by contradiction that $u$ verifies
\begin{equation}
\label{assurdo?}
\lim_{|x|\to\infty} u(x)\, |x|^{\alpha/(q-1)}=+\infty.
\end{equation}
This implies that for every $\varepsilon>0$ there exists a radius $R_\varepsilon$ such that
\[
|x|^{-\alpha}\, <\varepsilon\, u(x)^{q-1},\qquad |x|\ge R_\varepsilon.
\]
By taking $\varepsilon=a/2$, we thus get that $u$ verifies
\[
-\Delta u+\frac{a}{2}\, u^{q-1}\le 0,\qquad |x|\ge R_{\varepsilon}.
\]
Thus by \eqref{weakdecay} we get that $u=o(|x|^{-\gamma})$ for every $\gamma>0$, as $|x|$ goes to $\infty$. This clearly contradicts \eqref{assurdo?}, thus 
\[
0\le \liminf_{|x|\to \infty} u(x)\, |x|^{\alpha/(q-1)}<+\infty.
\]
This implies that there exists a sequence $\{r_k\}_{k\in\mathbb{N}}$ of radii converging to $\infty$ and a constant $A>0$ such that
\[
u(x)\le A\, |x|^{-\alpha/(q-1)},\qquad \mbox{ for }\ k \in\mathbb{N}\ \mbox{ and } |x|=r_k.
\]
We now take $\widetilde {u}(x)=\widetilde A\, |x|^{-\alpha/(q-1)}$, where $\widetilde A\ge A$ is a constant large enough such that there exists a radius $\widetilde R\gg 1$ for which
\[
-\Delta \widetilde u+a\, \widetilde u^{q-1}\ge b\, |x|^{-\alpha},\qquad \mbox{ for } |x|\ge \widetilde R. 
\]
We take $k_0=\min\{k\, :\, r_k\ge \widetilde R\}$, then we claim that
\begin{equation}
\label{hamel}
u(x)\le \widetilde u(x),\qquad \mbox{ for } r_{k}\le |x|\le r_{k+1}\quad \mbox{ and }\quad k\ge k_0. 
\end{equation}
If \eqref{hamel} were not true, there would exist a radius $r_k$ such that
\[
\widetilde u(y)-u(y):=\min_{r_k\le |x|\le r_{k+1}} \big(\widetilde u(x)-u(x)\big)<0.
\]
Since by construction we have 
$$
0\le \widetilde u(x)-u(x)\ \mbox{ for } |x|=r_k\qquad \mbox{ and }\qquad  0\le \widetilde u(x)-u(x)\ \mbox{ for } |x|=r_{k+1},
$$ 
then $y$ would be an interior minimum point of $\widetilde u-u$. By using this and the differential inequalities verified by $u$ and $\widetilde u$, we would get
\[
0\le \Delta \widetilde u(y)-\Delta\, u(y)\le a\, \big(\widetilde u(y)^{q-1}-u(y)^{q-1}\big)<0,
\]
thanks to the strict monotonicity of $t\mapsto t^{q-1}$. This gives the desired contradiction, thus \eqref{hamel} holds true and the decay estimate on $u$ is proved.
\end{proof}
\begin{oss}
Observe that the last part of the previous proof also shows that estimate \eqref{olè} is the best possible.
\end{oss}
In the previous proof we used the following result, which is essentially due to Strauss, see \cite[Radial Lemma 1]{St}. The statement is slightly more general (the original case corresponds to $q=2$), the proof just relies upon H\"older inequality. 
\begin{lm}[Strauss lemma]
\label{lm:strauss}
Let $N\ge 2$ and $u\in W^{1,2}_0(\mathbb{R}^N)\cap L^q(\mathbb{R}^N)$ be a radial function, where $0<q<\infty$. Then we have
\begin{equation}
\label{strauss}
|u(x)|^{2+q}\le S_{q,N}\, |x|^{-2\,(N-1)}\, \left(\int_{\mathbb{R}^N}
|\nabla u|^2\, dx\right)\,\left(\int_{\mathbb{R}^N} |u|^q\, dx\right),\qquad
x\in\mathbb{R}^N,
\end{equation}
where
\[
S_{q,N}=\left(\frac{1}{N\,\omega_N}\, \frac{2+q}{2}\right)^2
\]
and $\omega_N$ is the measure of the $N$-dimensional ball of radius $1$.
\end{lm}
\begin{proof}
Let $u\in C^\infty_0(\mathbb{R}^N)$ be a radial function, for every $p>1$ we have (with a small abuse of notation)
\[
|u(\varrho)|^p=-\int^{+\infty}_\varrho \frac{d}{dt} |u(t)|^p\, dt=-p\, \int_\varrho^{+\infty} u'(t)\, |u(t)|^{p-2}\, u(t)\, dt,\qquad \varrho>0.
\]
By taking $p=(2+q)/2$, we thus get
\[
\begin{split}
|u(\varrho)|^{(2+q)/2}&\le \frac{2+q}{2}\, \int_\varrho^{+\infty} |u'(t)|\, |u(t)|^{q/2}\, dt\\
&\le \frac{2+q}{2}\, \varrho^{1-N}\, \int_\varrho^{+\infty} \left(|u'(t)|\,t^{(N-1)/2}\right)\, \left(|u(t)|^{q/2}\, t^{(N-1)/2}\right)\, dt,\qquad \varrho>0.
\end{split}
\]
We now use H\"older inequality, then for every $\varrho>0$
\begin{equation}
\label{forte}
|u(\varrho)|^{(2+q)/2}\le \frac{2+q}{2}\, \varrho^{1-N}\, \left(\int_\varrho^{+\infty} |u'(t)|^2\,t^{N-1}\, dt\right)^{1/2}\, \left(\int_\varrho^{+\infty} |u(t)|^q\, t^{N-1}\, dt\right)^{1/2}.
\end{equation}
By noticing that for a radial function 
\[
\left(\int_{\mathbb{R}^N} |\nabla u|^2\, dx\right)^{1/2}=\sqrt{N\,\omega_N}\, \left(\int_0^{+\infty} |u'(t)|^2\,t^{N-1}\, dt\right)^{1/2},
\]
and
\[
\left(\int_{\mathbb{R}^N} |u|^q\, dx\right)^{1/q}=\left(N\,\omega_N\right)^{1/q}\, \left(\int_0^{+\infty} |u(t)|^q\, t^{N-1}\, dt\right)^{1/q}
\]
from \eqref{forte} we finally get \eqref{strauss} for smooth functions.
The inequality for $u\in W^{1,2}_0(\mathbb{R}^N)\cap L^q(\mathbb{R}^N)$ is obtained by a standard density argument.
\end{proof}

\begin{ack}
We warmly thank Fran\c{c}ois Hamel for some valuable suggestions on the proof of Lemma \ref{lm:decay}.
The work of the second author is part of the project 2010A2TFX2
{\it``Calcolo delle Variazioni''} funded by the Italian Ministry of
Research and University. Part of this work has been done during some visits of the first author to Pisa. The Departement of Mathematics of the University of Pisa and its facilities are kindly acknowledged.
\end{ack}

\end{document}